\documentclass[10pt]{article}
\usepackage[dvips]{graphicx}
\usepackage{psfrag}
\usepackage{amssymb,amsmath,amsfonts,euscript,amsthm}
\usepackage{color}
\usepackage{mathabx}
\usepackage{euscript}
\usepackage{mathrsfs}
\usepackage{appendix}
\usepackage[normalem]{ulem}
\DeclareMathAlphabet{\mathpzc}{OT1}{pzc}{m}{it}

\newtheorem{thm}{Theorem}[section]

\newtheorem{defn}[thm]{Definition}

\newtheorem{lemma}[thm]{Lemma}

\newcommand{\NN}{{\mathbb N}}

\renewcommand{\a}{\alpha}   \renewcommand{\b}{\beta}

   \newcommand{\e}{\epsilon}
\newcommand{\ve}{\varepsilon}
\newcommand{\g}{\gamma}      
\renewcommand{\k}{\kappa}   \renewcommand{\l}{\lambda}
\renewcommand{\L}{\Lambda}
\newcommand{\m}{\mu}        
\newcommand{\om}{\omega}
\newcommand{\p}{\phi}
    \newcommand{\vr}{\varrho}
\renewcommand{\r}{\rho}      \newcommand{\s}{\sigma}
   \renewcommand{\t}{\tau}
\renewcommand{\th}{\theta}      
     
\newcommand{\z}{\zeta}   \newcommand{\G}{\Gamma}
\newcommand{\D}{\Delta} \newcommand{\Sg}{\Sigma}

\newcommand{\diag}{\operatorname{diag}}

\newcommand{\osc}{\operatorname{osc.}}
\newcommand{\tu}{\operatorname{\bf u}}

\newcommand{\br}{\operatorname{\bf r}}

\newcommand{\TV}{\text{T.V.}}
\newcommand{\newvec}[1]{\vbox{\ialign{##\crcr$\displaystyle\rightharpoonup$\crcr\noalign{\kern-1pt\nointerlineskip}
$\hfil\displaystyle{#1}\hfil$\crcr}}} 
\newcommand{\longrightharpoonup}{\relbar\joinrel\rightharpoonup}
\newcommand{\longvec}[1]{\vbox{\ialign{##\crcr$\displaystyle\longrightharpoonup$\crcr\noalign{\kern-1pt\nointerlineskip}
$\hfil\displaystyle{#1}\hfil$\crcr}}} 

 \title{Global Transonic Solutions of Planetary Atmospheres in Hydrodynamic Region $\mbox{--}$ Hydrodynamic Escape Problem due to Gravity and Heat}
  \author{
  Bo-Chih Huang\footnote{Department of Mathematics, National Central University, Chung-Li 32001, Taiwan,
      E-mail: huangbz@math.ncu.edu.tw.},\ \
  Shih-Wei Chou\footnote{Department of Mathematics, National Central University, Chung-Li 32001, Taiwan,
      E-mail: swchou@math.ncu.edu.tw},\ \
  John M. Hong\footnote{Department of Mathematics, National Central University, Chung-Li 32001, Taiwan,
      E-mail: jhong@math.ncu.edu.tw.},\ \ and\
  Chien-Chang Yen\footnote{Department of Mathematics, Fu Jen Catholic University, New Taipei City 24205, Taiwan,
      E-mail: yen@math.fju.edu.tw.}
       \footnote{This work was partially supported by the Ministry of Science and Technology, R.O.C. under the grants MOST 104-2115-M-008-015-MY3, MOST 104-2115-M-008-005- and MOST 104-2115-M-030-004-.}}

\begin{document}
\maketitle


\date{}


\medskip

\begin{abstract}
The hydrodynamic escape problem (HEP), which is characterized by a free boundary value problem of Euler equation with gravity and heat, is crucial for investigating the evolution of planetary atmospheres. In this paper, the global existence of transonic solutions to the HEP is established using the generalized Glimm method. The new version of Riemann and boundary-Riemann solvers, are provided as building blocks of the generalized Glimm method by inventing the contraction matrices for the homogeneous Riemann (or boundary-Riemann) solutions. The extended Glimm-Goodman wave interaction estimates are investigated for obtaining a stable scheme and positive gas velocity, which matches the physical observation. The limit of approximation solutions serves as an entropy solution of bounded variations. Moreover, the range of the hydrodynamical region is also obtained.
\\
\\
MSC: 35L50, 35L60, 35L65, 35L67, 76N10, 85A20, 85A30\\
\\
Keywords: hydrodynamic escape problem; nonlinear hyperbolic systems of balance laws; generalized Riemann and boundary-Riemann problems; generalized Glimm scheme; hydrodynamic region.
\end{abstract}

\section{Introduction}
\setcounter{section}{1}

Spacecraft exploration of the planets in our solar system and the discovery of exoplanets has attracted considerable attention in the atmospheric escape from planetary objects~\cite{Erwin2013}. The Cassini spacecraft currently improves our understanding of the atmospheric escape from Titan~\cite{Guo2005}. The Maven Mission circuits around Mars for studying its atmospheric composition~\cite{Lin2012}. In July 2015, the New Horizons (NH) spacecraft completed its flyby of Pluto and discovered flowing ice and an extended haze on the planet. Pluto already exhibits a planetary geology that comprises flowing ices, exotic surface chemistry, mountain ranges, and vast haze. Analyzing Pluto's atmosphere reveals that Pluto's surface has a reddish hue, a simple hydrocarbon in its atmosphere, and the temperature for hazes to form at altitudes higher than 30 kilometers above Pluto's surface.

The hydrodynamic escape problem (HEP) is crucial for investigating of the evolution of planetary atmospheres. The HEP for a single-constituent atmosphere is governed by the following Euler equations with gravity and heat:
\begin{equation}
\label{3dHEP}
\left\{
\begin{split}
&\partial_t\r+\nabla\cdot(\r\tu)=0, \\
&\partial_t(\r\tu)+\nabla\cdot((\r\tu)\otimes\tu)+\nabla P=-\frac{GM_p\r}{|\br|^3}\br, \\
&\partial_tE+\nabla\cdot((E+P)\tu)=-\frac{GM_p\r}{|\br|^3}(\tu\cdot\br)+q,
\end{split}\right.
\end{equation}
where $\br$ is the position vector from the center of the planet to the particle of the gas; $\r,\ \tu,\ P$, and $E$ represent the density, velocity, pressure, and total energy of the gas respectively; and $G,\ M_p,\ q=q(\br)$ are the gravitational constant, mass of the planet, and heating, respectively.

In this paper, we are concerned with the three-dimensional inviscid hydrodynamic equations without thermal conduction in spherical symmetric space-time models, that is, we considered \eqref{3dHEP} to be of the following form
\begin{equation}
\label{HEP1}
\left\{
\begin{split}
&\partial_t\big(\rho x^2\big)+\partial_x\big(\rho u x^2\big)=0, \\
&\partial_t\big(\rho u x^2\big)+\partial_x\big(\rho u^2 x^2+P x^2\big)=-GM_p\rho+2Px, \quad 0<x_B<x<x_T, \\
& \partial_t\big(E x^2\big)+\partial_x\big((E+P)u x^2\big)=-GM_p\rho u+qx^2.
\end{split}\right.
\end{equation}
Here, $x$ denotes the distance from the center of the planet, $x_B$ and $x_T$ are the altitudes of the inner and outer boundaries of the planetary atmosphere, respectively. Typically, $x_B$ and $x_T$ are the altitudes of the upper thermosphere and exobase. The total energy $E$ is the sum of the kinetic energy and the internal energy of the gas flow,
\begin{equation}
\label{totalE}
E=\frac{1}{2}\rho u^2+\rho e=\frac{1}{2}\rho u^2+\frac{P}{\gamma-1},
\end{equation}
where $\gamma$ is the adiabatic constant with $1<\g<5/3$.

The steady transonic solutions of \eqref{HEP1} are crucial because of an almost hydrodynamic equilibrium state near the bottom boundary. The hybrid fluid/kinetic model \cite{TEDVJ2012} seems to be realistic approach, which comprises the hydrodynamic escape result \cite{ST2008B} and a drifting Maxwell-Boltzmann distribution function that includes the bulk velocity $u$ in the velocity distribution \cite{V2011, Y2004}. Tian and Toon \cite{TI2005B} implemented a numerical simulation using a time-dependent hyperbolic system. A time-independent model experiences singularity at the sonic points \cite{ST2008B}. For the relaxation methods in \cite{MurrayClay2009} for free conduction, the achieved numerical solutions depend on close guess of initial data with positive velocity. The first theoretical analysis for a steady HEP was reported in \cite{HYH}. Using the geometric singular perturbation method, the authors constructed smooth transonic stationary solutions issuing from subsonic states to supersonic states and various types of discontinuous stationary solutions for \eqref{HEP1}. For the time-evolutionary case, the global existence results are yet to be established. In this paper, the global existence of time-evolutionary transonic solutions to the HEP in the hydrodynamic region $\Sg\equiv\{(x,t):x_B\le x\le x_T,\ t\in[0,\infty)\}$ is established. The gravity and heat affecting intensity can be distinguished during the wave interaction, leading us to the effective development of the numerical simulation.

We define the notations as follows:
\begin{equation}
\label{setting1}
m:=\rho u,\quad U:=(\rho,\; m,\; E)^T.
\end{equation}
Using \eqref{totalE}, we can rewrite \eqref{HEP1} in a compact form
\begin{equation}
\label{3x3system}
U_t+f(U)_x=h(x)g(x,U),
\end{equation}
where $h(x)=-2/x$ and
\begin{equation}
\label{setting}
\begin{split}
f(U)&=\Big(m,\ \frac{3-\g}{2}\frac{m^2}{\r}+(\g-1)E,\ \frac{m}{\r}\Big(\g E-\frac{\g-1}{2}\frac{m^2}{\r}\Big)\Big)^T, \\
g(x,U)&=\Big(m,\ \frac{m^2}{\r}+\frac{GM_p}{2x}\r,\ \frac{m}{\r}\Big(\g E-\frac{\g-1}{2}\frac{m^2}{\r}\Big)+\frac{GM_p}{2x}m-\frac{xq}{2}\Big)^T.
\end{split}
\end{equation}
The complete model of the HEP is given by the following free boundary value problem:
\begin{eqnarray}
\label{FBVP}
\text{(HEP)}&&\left\{
\setlength\arraycolsep{0.1em}
\begin{split}
&U_t+f(U)_x=h(x)g(x,U),\quad (x,t)\in\Sg\equiv[x_B,x_T]\times[0,\infty), \\
&U(x,0)=U_0(x)\in\Omega,\ x\in[x_B,x_T], \\
&\r(x_B,t)=\r_B(t),\ m(x_B,t)=m_B(t),\quad t>0, \\
&\r\Big|_{\Sg},\ \frac{m}{\r}\Big|_{\Sg}>0,\ \mathfrak{Kn}(U)\Big|_{\Sg}\le 1,
\end{split}\right.
\end{eqnarray}
where the exobase of the atmosphere $x=x_T$ (as well as $\Sg$) must be determined and $\mathfrak{Kn}(U)$ denotes the Knudsen number of $U$. Physically, the region $\Sg$ is called the hydrodynamic region of \eqref{FBVP}. The position of the inner boundary $x_B$ may be probed through astronomical observation. However, determining the outer boundary $x_T$ is usually difficult due to the transition of the kinetic and hydrodynamical regions. Determining the position of the outer boundary $x=x_T$ and solving \eqref{3x3system} in $\Sg$ simultaneously is basically a free boundary problem, which makes it difficult to establish the global existence result. To overcome this difficulty, we first propose the following associated initial-boundary value problem (IBVP) without vacuum in $\Pi\equiv[x_B,\infty)\times[0,\infty)$:
\begin{eqnarray}
\label{IBVP}
\text{(IBVP)}&&\left\{
\setlength\arraycolsep{0.1em}
\begin{split}
&U_t+f(U)_x=h(x)g(x,U),\quad (x,t)\in\Pi\equiv[x_B,\infty)\times[0,\infty), \\
&U(x,0)=U_0(x)\in\Omega, \\
&\r(x_B,t)=\r_B(t),\ m(x_B,t)=m_B(t),\quad t>0,
\end{split}\right.
\end{eqnarray}
where $U_0(x)=(\rho_0(x),m_0(x),E_0(x))^T$ and $\Omega$ is an open domain centered at some sonic state
\begin{equation*}
U_s\equiv(\rho_s,m_s,E_s)\in \CMcal{T}:=\Big\{(\rho ,m,E)\Big|\; m=\r\sqrt{\g(\g-1)\Big(\frac{E}{\r}-\frac{u^2}{2}\Big)}\Big\}.
\end{equation*}
We call the set $\CMcal{T}$ the transition surface or the sonic states. The vacuum case is excluded from this formula because the atmospheric gas does not act as fluid when the density tends to zero. Whether the Glimm method can be applied to the vacuum case for the general system has remained unsolved for decades. In this paper, a new version of the Glimm method is used for establishing the existence of global entropy solutions of \eqref{IBVP} under the following conditions:
\begin{enumerate}
\item [($A_1$)] $\rho_0(x)$, $m_0(x),$ $E_0(x)$, $\r_B(t)$ and $m_B(t)$ are bounded positive functions with small total variations, and there exists $\vr>0$ sufficiently small such that $\r_0(x)\ge\vr$ and $\r_B(t)\ge\vr$ for $(x,t)\in\Pi$;
\item [($A_2$)] $\min\limits_{t\ge 0}\{m_B(t)\}>(1+\e)\TV\{U_0(x)\}+(1+\e+\e^2)^2\CMcal{C}$ for $0<\e<\frac{1}{2}$ and some positive constant $\mathcal{C}$;
\item [($A_3$)] $q(x)\in W^{1,1}[x_B,\infty)$.
\end{enumerate}
Under the condition ${\rm(}A_1{\rm)}\sim{\rm(}A_3{\rm)}$, \eqref{IBVP} consists of global entropy solutions with positive velocity in $\Pi$ (Main Theorem I). In addition, under a certain constraint of transonic initial data, in the complement of $\Sg$, denoted as $\Sg^c$, the wave speeds of Glimm's approximate solutions $\{U_{\D x}\}$ to \eqref{IBVP} are positive, that is, the entropy flow $U$ to which $\{U_{\D x}\}$ converges is supersonic in $\Sg^c$ so that the waves of $U$ in $\Sg^c$ do not move into $\Sg$ to interact with the waves in $\Sg$. Moreover, we prove that the Knudsen number $\mathfrak{Kn}(U)$ of $U$ in $\Sg$ satisfies
$$
\mathfrak{Kn}(U)\le 1,
$$
which implies that $\Sg$ fulfills the physical meaning of the hydrodynamic region \cite{ZSE}. Using this strategy, we can prove that the solution $U\big|_{\Sg}$ for \eqref{IBVP} is indeed the entropy solution of \eqref{FBVP} (Main Theorem II).

Let us review some previous results related to this topic and clarify the motivation of the study. When $g\equiv 0$, the system \eqref{3x3system} is reduced to the strictly hyperbolic system,
\begin{equation}
\label{homo}
U_t+f(U)_x=0.
\end{equation}
The entropy solution to the Riemann problem was first constructed by Lax \cite{LA}. In particular, the solution is self-similar and consists of constant states separated by elementary waves: rarefaction waves, shocks, and contact discontinuities. Furthermore, the global existence of weak solutions to the Cauchy problem was established by Glimm \cite{G}, who considered Lax's solutions as the build blocks of the scheme. For the inhomogeneous hyperbolic systems,
\begin{eqnarray}
\label{balance}
 U_t +f(x,U)_x = g(x,U).
\end{eqnarray}
The Cauchy problem was first studied by Liu \cite{TP1}. For the Cauchy problem of the general quasi-linear, strictly hyperbolic system
\begin{eqnarray}
\label{qlinear}
U_t +f(x,t,U)_x = g(x,t,U).
\end{eqnarray}
The existence of entropy solutions was first established by Dafermos and Hsiao \cite{DH}. In \cite{CHS2,HL,LR}, system \eqref{qlinear} was studied under dissipative conditions by using the asymptotic expansion of the classical Riemann solutions. The aforementioned conditions contribute considerably in investigating the systems \eqref{balance} and \eqref{qlinear}. In the aforementioned studies, the source term was used for generating an extra stationary characteristic field in the Riemann problem. The time-independent wave curves generated by this filed are tangential to the classical 1-wave curves at sonic states, resulting in the nonuniqueness of solutions of a Riemann problem. The total variations of solutions may blow up in a finite time and the blow up phenomenon can be eliminated using a further dissipative assumption.

When $f$ and $g$ are independent of $x$, Bianchini-Bressan \cite{BB} studied the existence result by using regularization method. Luskin-Temple \cite{LT} and the authors in \cite{CHS} establish the existence result by combining Glimm's scheme with the method of fractional steps. Base on these studies, we can consider the effect of the source in \eqref{3x3system} as the perturbations of the solutions to the homogeneous conservation laws. The appearance of the source terms in \eqref{3x3system} breaks up the self-similarity of Riemann waves. But the effect is only up to $O(1)\D t$, that is, the effect of the source terms on self-similar waves is of the order $\D t$ in each Riemann cell. More precisely,
\begin{equation}
\label{1.12t}
(df(U)-\xi I_3)\dot{U}=(t-t_0)h(x)g(x,U)\approx(\D t)h(x)g(x,U).
\end{equation}
Therefore, we can construct the approximate solution for our generalized Riemann problem as $U=\widetilde{U}+\widebar{U}$ in each Riemann cell, where $\widetilde{U}$ solves \eqref{homo} and $\widebar{U}$ solves the linearized system of \eqref{3x3system} around $\widetilde{U}$. We have
$$
|\widebar{U}|=O(1)|\widetilde{U}|,
$$
and
$$
U(x,t)=S(x,t,\widetilde{U})\cdot\widetilde{U},
$$
for some contraction matrix $S(x,t,\widetilde{U})$ depending on $\widetilde{U}$ (as well as $f,\ g$). This construction of the generalized Riemann solver is in contrast to the fractional step scheme \cite{LT} and other operator splitting methods \cite{BB,CHS,GST}. For instance, in \cite{CHS}, the effect of the source on the solutions of the classical Riemann problem is decoupled. In our case, the effect of the source on the solutions of the classical Riemann problem is strongly coupled. The estimates of wave interaction are more complicated than in \cite{CHS}.

For the stability of the generalized Glimm scheme, contrary to the methods used in \cite{CHS,CHS1,CHS2,DH,H,HL,HT2,LT}, in which the positivity of the gas velocity is assumed, we can demonstrate, through the structure of the generalized Glimm's approximate solution, that the escape velocity of the gas is globally positive, which matches the astronomical observation. Therefore, the uniform bounds of the total variations of the approximate solutions in \eqref{IBVP} can be achieved by showing that
(1) the Glimm functionals of $\widetilde{U}$ are nonincreasing in time and
(2) the perturbations have a uniform bound of the total variations in each time step.
In addition, we prove positivity of the gas velocity through a rigorous mathematical proof. Based on the contraction matrix $S$, we can achieve a more accurate formula of wave interaction estimates that lead to the decay result of the Glimm functionals, and a new relation between the velocity and the Glimm functionals in each time step. Consequently, the stability of the generalized Glimm scheme and the global positivity of the gas velocity is obtained, as shown in Section 3. Based on ${\rm(}A_1{\rm)}\sim{\rm(}A_3{\rm)}$ and the estimation of interaction, the existence of global entropy solutions in transonic gas flow without any dissipative condition is established.

We now introduce the definitions of weak solutions and entropy solutions for \eqref{IBVP}, and state the main theorems.
\begin{defn}
Consider the initial-boundary value problem in \eqref{IBVP}. A measurable function $U(x,t)$ is a weak solution of \eqref{IBVP} if
\begin{equation*}
\iint_{x>x_B, t>0}\left\{U\phi_t+f(U)\phi_x+h(x)g(x,U)\phi\right\}dxdt +\int^{\infty}_{x_B} U_0(x)\phi(x,0)
+\int^{\infty}_{0}f(U(x_B,t))\phi(x_B,t)dt=0,
\end{equation*}
for any test function ${\phi} \in C^1_0(\Pi)$.
\end{defn}
\begin{defn}
Let $\Omega$ be a convex subset of $\mathbb{R}^3$. A pair $(\eta(U),\om(U))$ is an entropy pair of \eqref{HEP1} if $\eta$ is convex on $\Omega$ and
\begin{eqnarray*}
d\om=d\eta df\quad\text{on}\quad\Omega.
\end{eqnarray*}
Furthermore, a measurable function $U$ is an entropy solution of \eqref{IBVP} if $U$ is a weak solution of \eqref{IBVP} and satisfies
\begin{equation}
\label{ibvpentropy}
\iint_{x>x_B,t>0}\left\{\eta\phi_t+\om\phi_x +d\eta\cdot hg\phi\right\}dxdt+\int^\infty_{x_B}\eta(U_0(x))\phi(x,0)dx
+\int^{\infty}_{0}\om(U(x_B,t))\phi(x_B,t)dt\geq 0,
\end{equation}
for every entropy pair $(\eta(U),\om(U))$ and any positive test function $\phi\in C^1_0(\Pi)$.
\end{defn}

\noindent {\bf Main Theorem I.}
{\it Consider the initial-boundary value problem \eqref{IBVP} with transonic initial data $U_0=(\r_0,m_0,E_0)^T$. Assume that the inner boundary data $(\r_B(t),m_B(t))$ satisfies the condition ${\rm(}A_1{\rm)}\sim{\rm(}A_3{\rm)}$ and the heat $q$ satisfies the condition {\rm(}$A_3${\rm)}. Let $\{U_{\theta,\Delta x}\}$ be the sequence of approximate solutions of \eqref{IBVP} by using the generalized Glimm scheme. Then, there exist a null set $N \subset \Phi$ and a subsequence $\{\Delta x_{i}\}\rightarrow 0$ such that if $ \theta \in \Phi\setminus N$, then
$$
U(x,t):=\lim_{\scriptstyle \Delta x_i\rightarrow 0} U_{\theta,\Delta x_i}(x,t)
$$
is the positive entropy solution of \eqref{IBVP}. In particular, the gas velocity is positive in $\Pi$}.
\medskip \\
\noindent {\bf Main Theorem II.}
{\it Assume that the transonic initial data $U_0=(\r_0,m_0,E_0)^T$ such that $\r_0,E_0$ is decreasing and $m_0$ is increasing and $u_0(x_B)<c_0(x_B)$, where $u_0,c_0$ as defined in \eqref{2.5.5}.  There exists $x_T>x_B$ depending on the initial and boundary data such that $\Sg\equiv [x_B,x_T]\times[0,\infty)$ is the hydrodynamic region of \eqref{FBVP}, which means $U(x,t)|_{\Sg}$ is the global entropy solution of \eqref{FBVP} satisfying $\mathfrak{Kn}(U)\big|_{\Sg}\le 1$.}\\

This paper is organized as follows. Section 2 presents the generalized solvers for Riemann and boundary-Riemann problems based on the construction of the approximate solutions to these problems through the asymptotic expansion and operator splitting techniques. The residuals of the solutions in each grid are calculated to preserve the consistency of the proposed scheme. Section 3 presents a generalized version of the Glimm scheme. Moreover, the generalized wave interaction estimate, nonincreasing Glimm functional, and estimate for the total variation of the perturbations in each time strip are obtained. The global existence of the entropy solutions for \eqref{HEP1} is proved. The hydrodynamic region is determined in the final section.

\section{Generalized solutions for the Riemann and Boundary-Riemann problems}
\setcounter{equation}{0}

In this section, we introduce a new method of constructing the approximate solutions to the Riemann and boundary-Riemann problems of \eqref{3x3system}, which are the building blocks of the generalized Glimm scheme of the HEP. The residuals of the approximate solutions will be estimated for maintaining the consistency of the generalized Glimm scheme. Let us select the spatial resolution $\Delta x>0$ and the temporal step $\Delta t>0$ sufficiently small, which satisfies the Courant-Friedrichs-Lewy (CFL) condition
\begin{equation}
\label{CFL}
\l_*:=\frac{\Delta x}{\Delta t} > \sup\limits_{(\rho,m,E)\in\Omega}\left\{\frac{m}{\rho}+\sqrt{\g(\g-1)\Big(\frac{E}{\r}-\frac{m^2}{2\r^2}\Big)}\right\}.
\end{equation}
We define the inner region at the location $x_0$ and time $t_0$
\begin{equation}
\label{ingrid}
D(x_0,t_0):=\{(x,t)\mid|x-x_0|<\Delta x, \;t_0<t< t_0+\Delta t\},
\end{equation}
and the boundary region at the lower boundary $x_B$ and time $t_0$
\begin{equation}
\label{bdgrid}
D(x_B,t_0):=\{(x,t)\mid x_B<x<x_B+\Delta x, \;  t_0<t< t_0+\Delta t\}.
\end{equation}
The Riemann problem of \eqref{HEP1} in $D(x_0,t_0)$, denoted by $\CMcal{R}_G(x_0,t_0; g)$, is given by
\begin{equation}
\label{RP}
\CMcal{R}_G(x_0,t_0;g) :\qquad
\left\{\begin{array}{l}
U_t+f(U)_x=h(x)g(x,U),\quad (x,t)\in D(x_0,t_0), \\
U(x,t_0)=\left\{\begin{array}{ll}
U_L, &\text{if }x_0-\D x<x<x_0, \\
U_R, &\text{if }x_0<x<x_0+\D x,
\end{array}\right.
\end{array}\right.
\end{equation}
and the boundary-Riemann problem of \eqref{HEP1} in $D(x_B,t_0)$, denoted by
$\CMcal{BR}_G(x_B,t_0; g)$ can be expressed as
\begin{equation}
\label{BRP}
\CMcal{BR}_G(x_B,t_0; g):\qquad
\left\{\begin{array}{ll}
U_t+f(U)_x=h(x)g(x,U),         &(x,t)\in D_B(x_B,t_0), \\
U(x,t_0)=U_R,                  &x_B\le x\le x_B+\D x, \\
\r(x_B,t)=\r_B,\ m(x_B,t)=m_B, &t_0<t<t_0+\D t,
\end{array}\right.
\end{equation}
where $m,\ U,\ f$, and $g$ are defined in \eqref{setting1} and \eqref{setting}; $U_L=(\rho_L, m_L,E_L)$ and $U_R=(\rho_R, m_R,E_R)$ are the left and right constant states; and $\r_B>0$ and $m_B>0$ are the density and momentum at the boundary $x_B$, respectively. By setting the source term $g \equiv 0$ in \eqref{RP} and \eqref{BRP}, the corresponding classical Riemann and boundary-Riemann problems are denoted by $\CMcal{R}_C(x_0,t_0)=\CMcal{R}_G(x_0,t_0;0)$ and $\CMcal{BR}_C(x_B,t_0)=\CMcal{BR}_G(x_B,t_0;0)$.

\subsection{Construction of approximate solutions to the Riemann and boundary-Riemann problems}

The system \eqref{3x3system} is a strictly hyperbolic system whose Jacobian matrix $df$ has three distinct real eigenvalues:
$$
\lambda_1(U):= u-c(U),\quad \l_2(U):=u,\quad\text{and}\quad\lambda_3(U):=u+c(U),
$$
where
\begin{equation}
\label{2.5.5}
u=\frac{m}{\r}\quad\text{and}\quad
c(U)=\sqrt{\g(\g-1)\Big(\frac{E}{\r}-\frac{u^2}{2}\Big)}.
\end{equation}
The corresponding right eigenvectors of $df$ are
$$
R_1(U)=(-1,c-u,uc-H)^T,\quad R_2(U)=(1,u,\tfrac{1}{2}u^2)^T,\quad\text{and}\quad R_3(U)=(1,c+u,uc+H)^T
$$
where the total specific enthalpy $H$ is
\begin{equation}
\label{enthalpy}
H=H(U)=\frac{\g E}{\r}-\frac{\g-1}{2}u^2,
\end{equation}
and $c(U)$ is the {\it sound speed} of the gas. Here, the gas is assumed to be ideal so that the pressure satisfies
\begin{equation}
\label{ideal}
P=\r RT,
\end{equation}
where $R$ is the molar gas constant and $T$ is the absolute temperature. According to \eqref{totalE} and \eqref{ideal}, the sound speed can be expressed as
\begin{equation*}
c=\sqrt{\g RT}.
\end{equation*}
Since the laws of thermodynamics indicate that the absolute zero temperature cannot be reached by only the thermodynamic process, it allows us to assume that
\begin{equation}
\label{minc}
\min_{U\in\Omega}c(U)=c_*>0.
\end{equation}
Furthermore, we have
\begin{equation*}
\nabla\lambda_i(U)\cdot R_i(U)=\frac{(\g+1)c}{2\r}>0,\ i=1,3,\quad\text{and}\quad\nabla\l_2(U)\cdot R_2(U)=0,
\end{equation*}
which implies that the first and third characteristic fields are genuinely nonlinear and the second characteristic field is linearly degenerate. Therefore, the entropy solutions for $\CMcal{R}_C(x_0,t_0)$ and $\CMcal{BR}_C(x_B,t_0)$ consist of either shock waves, rarefaction waves first or third characteristic fields, or contact discontinuities from the second characteristic field. For each $i=1,3$, the $i$-rarefaction wave is a self-similar function
$$
U=U(\xi), \quad \xi=\frac{x-x_0}{t-t_0},
$$
which satisfies
\begin{equation*}
(df(U)-\xi I_3)\cdot \frac{dU}{d\zeta}=0,
\end{equation*}
where $I_3$ is the $3\times3$ identity matrix and the admissible $i$-shock is a discontinuous function satisfying the Rankine-Hugoniot condition
\begin{equation}
\label{jumpcond}
 s [U]=[f(U)],
\end{equation}
and Lax's entropy condition
\begin{equation*}
\lambda_i(U_R) <s< \lambda_i(U_L),
\end{equation*}
where $s$ is the speed of the shock-front and $[\cdot]$ denotes the difference of states across the shock. According to Lax's method \cite{LA}, we can obtain the existence and uniqueness of the entropy solution for $\CMcal{R}_C(x_0,t_0)$. The solution consists of at most four constant states separated by shocks, rarefaction waves, or contact discontinuity. However, for $\CMcal{BR}_C(x_B,t_0)$, even under the Rankine-Hugoniot and Lax's entropy conditions, we may not obtain the uniqueness of the weak solutions when $U_R$ is near the transition surface $\CMcal{T}$, see Figure 1.
\begin{center}
\includegraphics[scale=0.8]{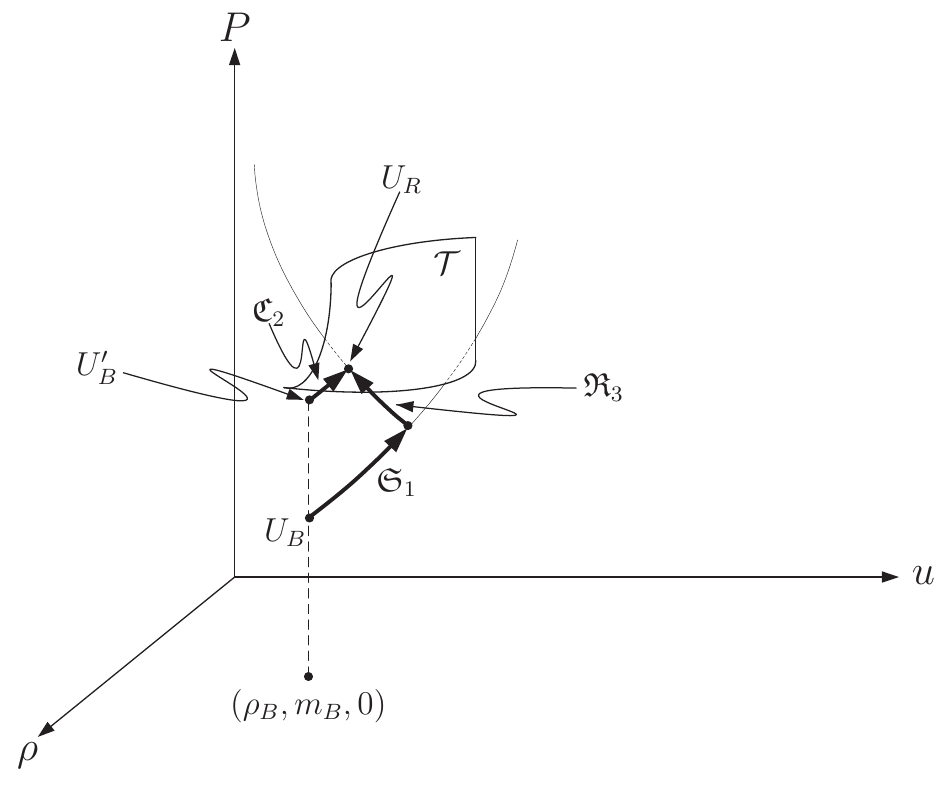}\\
Figure 1. Two states $U_B$ and  $U_B'$ connect to $U_R\in\mathcal{T}$ by two different waves: \\
$\mathfrak{S}_2+\mathfrak{R}_3$  and $\mathfrak{C}_2$. Both states satisfy the Rankine-Hugoniot and Lax-entropy conditions.\medskip
\end{center}
Moreover, the total variation of these solutions can be large even $|\r_R-\r_B|$, $|m_R-m_B|$ are small. To solve this problem, we impose an additional condition on the solutions:
\begin{enumerate}
\item[($\CMcal{E}$)] A weak solution $U=(\rho, m,E)$ is the entropy solution of $\CMcal{BR}_C(x_B,t_0)$ if $U$ has the least total variation in $\rho$ within all weak solutions of $\CMcal{BR}_C(x_B,t_0)$.
\end{enumerate}

Under the condition $(\CMcal{E})$, we can select the unique entropy solution for $\CMcal{BR}_C(x_B,t_0)$. In addition, the entropy solution does not consist of the 0-speed shock from the first characteristic field attached on the boundary $x=x_B$. The following theorem states the existence and uniqueness of entropy
solutions for $\CMcal{R}_C(x_0,t_0)$ and $\CMcal{BR}_C(x_B,t_0)$.

\begin{thm}
{\rm \cite{HYH,IT1,IT2,S2}} Suppose $U_L \in\Omega$. Then, there is a neighborhood $\Omega_1\subset\Omega$ of $U_L$ such that if $U_R\in\Omega_1$, $\CMcal{R}_C(x_0,t_0)$ has a unique solution consisting of at most four constant states separated by shocks, rarefaction waves, and contact discontinuity. Moreover, under the additional condition $(\CMcal{E})$, there exist a neighborhood $\Omega_2\subset\Omega$ of $U_R$ and $E_B>0$ exist such that $U_B=(\rho_B, m_B,E_B)\in\Omega_2$ and $\CMcal{BR}_C(x_B,t_0)$ admits a unique self-similar solution $U$ satisfying $U(x_B,t)=U_B$.
\end{thm}

We can now construct the approximate solutions for $\CMcal{R}_G(x_0,t_0;g)$ and $\CMcal{BR}_G(x_B,t_0;g)$ by using Theorem 2.1. Let $\widetilde{U}=(\tilde{\r},\tilde{m},\widetilde{E})^T$ be the entropy solution of $\CMcal{R}_C(x_0,t_0)$. Then, for $(x,t)\in D(x_0,t_0)$, the approximate solution $U$ of $\CMcal{R}_G(x_0,t_0;g)$ is given by
\begin{equation}
\label{aproxsol}
U(x,t)=\widetilde{U}(x,t)+\widebar{U}(x,t),
\end{equation}
where $\widebar{U}(x,t)$ is constructed using the following steps: (1) Linearizing of the system \eqref{3x3system} around the homogeneous solution $\widetilde{U}$. (2) Averaging the coefficient of the linearized system. (3) Applying the operator-splitting method to the modified system. The detailed construction of $\widebar{U}(x,t)$ is provided in appendix A. On the basis of these steps and \eqref{aproxsol} and \eqref{perturbsol}, the approximate solution $U(x,t)$ can be expressed as
\begin{equation}
\label{aproxsol4}
U(x,t)=S(x,t,\widetilde{U})\widetilde{U},
\end{equation}
where
\begin{equation}
\label{solver}
S(x,t,\widetilde{U})=\left[\begin{array}{ccc}
e^{h\tilde{u}t}\cosh(hvt) & 0 & 0 \\
\vspace{-0.4cm} \\
ve^{h\tilde{u}t}\sinh(hvt) & e^{h\tilde{u}t}\cosh(hvt) & 0 \\
S_{31} & S_{32} & S_{33}
\end{array}\right],
\end{equation}
and $v=v(x)$ is defined as
\begin{equation}
\label{grav}
v(x)=\sqrt{\frac{GM_p}{2x}},
\end{equation}
and
\begin{eqnarray*}
&&S_{31}=\frac{-xq}{2\g\tilde{m}}(e^{\g h\tilde{u}t}-1)-\frac{v^3}{v^2-(\g-1)^2\tilde{u}^2}
(ve^{\g h\tilde{u}t}-ve^{h\tilde{u}t}\cosh(hvt)-(\g-1)\tilde{u}e^{h\tilde{u}t}\sinh(hvt)), \\
&&S_{32}=\frac{(\g-1)\tilde{u}(v^2+(\g-1)\tilde{u}^2)}{2(v^2-(\g-1)^2\tilde{u}^2)}
(e^{\g h\tilde{u}t}-e^{h\tilde{u}t}\cosh(hvt))
+\frac{v(2v^2-(3\g-2)(\g-1)\tilde{u}^2)}{2(v^2-(\g-1)^2\tilde{u}^2)}e^{h\tilde{u}t}\sinh(hvt), \\
&&S_{33}=-\frac{(\g-1)(v^2+(\g-1)\tilde{u}^2)}{v^2-(\g-1)^2\tilde{u}^2}e^{\g h\tilde{u}t}
+\frac{\g v}{v^2-(\g-1)^2\tilde{u}^2}(ve^{h\tilde{u}t}\cosh(hvt)+(\g-1)\tilde{u}e^{h\tilde{u}t}\sinh(hvt)).
\end{eqnarray*}
Furthermore, through complex computation, we have
\begin{equation}
\label{1stubar}
\widebar{U}(x,t)
=-\frac{\tilde{\r}\D t}{x}\Big(2\tilde{u},2(\tilde{u}^2+v^2),\tilde{u}\Big(\tilde{u}^2+2v^2+\frac{2\tilde{c}^2}{\g-1}\Big)-\frac{q}{x\tilde{\r}}\Big)^T+O(1)(\D t)^2.
\end{equation}
Consequently,
\begin{equation}
\label{2.13.1}
|\widebar{U}|=|(S-I_3)\widetilde{U}|=O(1)(\Delta t),
\end{equation}
and $\widebar{U}\rightarrow 0$ as $\D t\to 0$ or $h\to 0$, which is consistent with the entropy solution for homogeneous conservation laws. Moreover, the approximation in \eqref{aproxsol4} continues to be true when $\widetilde{U}$ is a constant state.

The construction of the approximate solution for $\CMcal{BR}_G(x_B,t_0;g)$ is similar to that for $\CMcal{R}_G(x_0,t_0;g)$. Therefore, the approximate solution for $\CMcal{BR}_G(x_B,t_0;g)$ is also given by \eqref{aproxsol4}. The perturbation $\widebar{U}$ in $\CMcal{BR}_G(x_B,t_0;g)$ may not satisfy $\widebar{U}(x_B,t)=0$ because of \eqref{1stubar}, which means that the approximate solution $U$ may not match the boundary condition.
However by \eqref{2.13.1}, the error between the approximation $U$ and the boundary data $\widehat{U}_B:=(\r_B,m_B,0)^T$ can be estimated by
\begin{equation}
\label{2.14}
|U(x_B,t)-\widehat{U}_B|\le|\widebar{U}(x_B,t)|+|\widetilde{U}(x_B,t)-\widehat{U}_B|=O(1)(\Delta t)+\underset{D(x_B,t_0)}{\osc}\{\widehat{U}_B\},
\end{equation}
where $\underset{D(x_B,t_0)}{\osc}\{\widehat{U}_B\}$ denotes the oscillation of a function $\widehat{U}_B$ in the set $D(x_B,t_0)$. This indicates that such approximation does not affect the stability and consistency of the generalized Glimm method, which will be discussed later.

\subsection{Residuals of the approximate solutions for Riemann and boundary-Riemann problems}

To demonstrate the consistency of the generalized Glimm scheme, it is necessary to calculate the residuals of the approximate solutions for
Riemann and boundary-Riemann problems. Given a measurable function $U$,
region $\Gamma \subset\Pi$, and test function $\phi\in C^1_0(\Gamma)$, the residual of $U$ for \eqref{3x3system} in $\Gamma$ is defined as
\begin{equation}
\label{res1}
R(U, \Gamma,\phi):=\iint_{\Gamma}\left\{U\phi_t+f(U)\phi_x+h(x)g(x,U)\phi\right\}dxdt.
\end{equation}
We have the following estimates.
\begin{thm}
Let $U$ and $U^B$ be the approximate solutions of the Riemann problem $\CMcal{R}_G(x_0,t_0;g)$ and the boundary-Riemann problem $\CMcal{BR}_G(x_B,t_0;g)$ respectively. Let $\phi\in C^1_0(\G)$ be a test function. Suppose that $U=\widetilde{U}+\widebar{U}$ and $U^B=\widetilde{U}^B+\widebar{U}^B$. Then the residuals of $U$ and $U^B$ can be estimated respectively by
\begin{align}
\label{RPres}
&R(U,D(x_0,t_0),\phi)\nonumber\\
&=\int^{x_0+\Delta{x}}_{x_0-\Delta{x}}(U\phi)(x,t)\Big|^{t=t_0+\Delta t}_{t=t_0}dx
+ \int^{t_0+\Delta{t}}_{t_0}(f(U)\phi)(x, t)\Big|^{x=x_0+\Delta x}_{x=x_0-\Delta x}dt\nonumber\\
   &\quad+O(1)\left((\Delta t)^2(\Delta
   x)+(\Delta{t})^3+(\Delta{t})^2\underset{D(x_0,t_0)}{\osc}\{\widetilde{U}\}\right)\|\phi\|_\infty,
\end{align}
and
\begin{align*}
&R(U^B, D(x_B,t_0), \phi)\nonumber\\
&=\int^{x_B+\Delta{x}}_{x_B}(U^B\phi)(x, t)\Big|^{t=t_0+\D t}_{t=t_0}dx
+\int^{t_0+\Delta{t}}_{t_0}(f(U^B)\phi)(x, t)\Big|^{x=x_B+\Delta x}_{x=x_B}dt\nonumber\\
   &\quad+O(1)\left((\Delta t)^2(\Delta
   x)+(\Delta{t})^3+(\Delta{t})^2\underset{D(x_B,t_0)}{osc}\{\widetilde{U}^B\}\right)\|\phi\|_\infty,
\end{align*}
where $\underset{\Lambda}{osc}\{w\}$ denotes the oscillation of a function $w$ in the set $\Lambda$, and $D(x_0,t_0)$ and $D(x_B,t_0)$ are given by \eqref{ingrid} and \eqref{bdgrid}, respectively.
\end{thm}
\begin{proof}
We only demonstrate the calculation of  $R(U,D(x_0,t_0),\p)$; the calculation of $R(U^B,D(x_B,t_0),\p)$ is similar.
Without loss of generality, let $t_0=0,\ D:=D(x_0,t_0)$ with $x_0>x_B$, and $\widetilde{U}$ consists of the 1-shock with speed $s_1$, 2-contact discontinuity with speed $s_2$,
and 3-rarefaction wave with lower speed $s_3^-$ and upper speed $s^+_3$.
    By \eqref{aproxsol4} and \eqref{res1},
\begin{align}
\label{res2}
R(U, D, \phi)&=\iint_D\left\{\widetilde{U}\phi_t+f(\widetilde{U})\phi_x\right\}dxdt
+\iint_D\left\{\widebar{U}\phi_t+(f(U)-f(\widetilde{U}))\phi_x+h(x)g(x,U)\phi\right\} dxdt\nonumber\\
&\equiv Q_1+Q_2.
\end{align}
According to the structure of $\widetilde{U}$, $Q_2$ can be evaluated as
\begin{align}
\label{perturbres}
Q_2&=\left(\int^{\Delta{t}}_{0}\!\!\!\int^{x_0+s_1t}_{x_0-\Delta{x}}+\int^{\Delta{t}}_{0}\!\!\!\int^{x_0+s_2t}_{x_0+s_1t}
 +\int^{\Delta{t}}_{0}\!\!\!\int_{x_0+s_2t}^{x_0+s^-_3t}+\int^{\D t}_0\!\!\!\int_{x_0+s^-_3t}^{x_0+s^+_3t}
 +\int^{\Delta{t}}_{0}\!\!\!\int^{x_0+\Delta{x}}_{x_0+s^+_3t}\right)\nonumber\\
 &\qquad\left\{\widebar{U}\phi_t+(f(U)-f(\widetilde{U}))\phi_x+h(x)g(x,U)\phi\right\} dxdt\nonumber\\
 &\equiv Q_{21}+Q_{22}+Q_{23}+Q_{24}+Q_{25}.
\end{align}
\begin{center}
\includegraphics[scale=1]{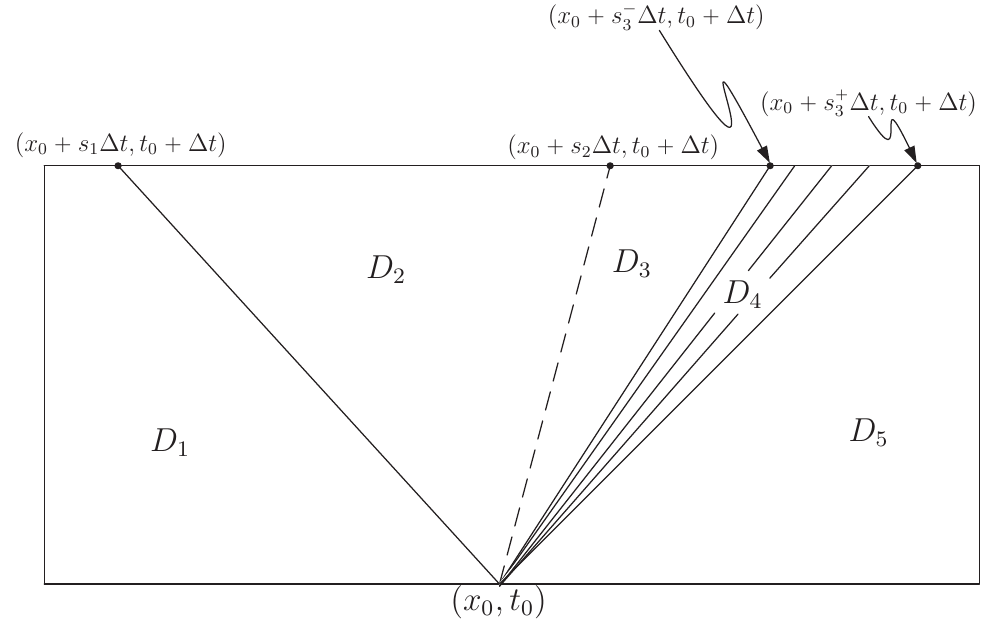}\\
Figure 2. Typical solution to the classical Riemann problem $\CMcal{R}_C(x_0,t_0)$.\medskip
\end{center}
Here $\widetilde{U}$ consists of different constant states in the region $D_i,\ i=1,2,3,5$, and self-similar in $D_4$, where $\{D_i\}_{i=1}^5$  is as shown in Figure 2.  Next, according to \eqref{aproxsol4} and \eqref{solver} and the complicate calculation,
\begin{equation}
\label{diff1}
\begin{split}
\widebar{U}_t-hg(x,U)
&=\partial_t((S-I_3)\widetilde{U})-hg(x,S\widetilde{U})=S_t\widetilde{U}-hg(x,S\widetilde{U}) \\
&=h\tilde{\r}v\big(0,ve^{h\tilde{u}t}\sinh(hvt)\tanh(hvt),\D_3\big)^T=O(1)(\D t)^2,
\end{split}
\end{equation}
where
$$
\D_3=\tanh(hvt)\cdot
\left\{
\begin{split}
&\frac{xq}{2\tilde{m}}(e^{\g h\tilde{u}t}-1)
+\frac{(v^2+(\g-1)\tilde{u}^2)\tilde{c}^2+\g v^4}{v^2-(\g-1)^2\tilde{u}^2}(e^{\g h\tilde{u}t}-e^{h\tilde{u}t}\cosh(hvt)) \\
&\quad-\frac{\tilde{u}v((\g-3)(\g-1)^2\tilde{u}^2+2\g\tilde{c}^2+(2\g^2-3\g+3)v^2)}{2(v^2-(\g-1)^2\tilde{u}^2)}e^{h\tilde{u}t}\sinh(hvt) \\
&\quad+\frac{(\g-1)v^2}{2}e^{h\tilde{u}t}\sinh(hvt)\tanh(hvt)
\end{split}\right\}.
$$
Since $\sinh(hvt)=O(1)\Delta t$, $\tanh(hvt)=O(1)\Delta t$, and $e^{\gamma h \tilde{u}t}-1=O(1)\Delta t$ when $0<t<\Delta t$. Applying Green's theorem and \eqref{diff1} to $Q_{2i},\ i=1,2,3,5$, we have
\begin{align}
\label{Q21}
Q_{21}&=\int^{x_0+s_1\Delta{t}}_{x_0-\Delta{x}}(\widebar{U}\p)(x,\Delta{t})dx-
\int^{x_0}_{x_0-\Delta{x}}(\widebar{U}\p)(x,0)dx
-\int^{\Delta{t}}_{0}s_1(\widebar{U}\p)(s_1t-,t)dt\nonumber\\
&\quad+\int^{\Delta{t}}_{0}\big((f(U)-f(\widetilde{U}))\p\big)(x, t)\Big|^{x=s_1t-}_{x=-\Delta{x}}dt+O(1)(\Delta t)^3(\Delta x)\|\phi\|_\infty , \\
\label{Q22}
Q_{22}&=\int^{x_0+s_2\Delta{t}}_{x_0+s_1\Delta{t}}(\widebar{U}\p)(x,\Delta{t})dx
+\int^{\Delta{t}}_{0}\big((f(U)-f(\widetilde{U})-s_2\widebar{U})\p\big)(s_2t-, t)dt\nonumber\\
&\quad-\int^{\Delta{t}}_{0}\big((f(U)-f(\widetilde{U})-s_1\widebar{U})\p\big)(s_1t+, t)dt+O(1)(\Delta t)^3(\Delta x)\|\phi\|_\infty, \\
\label{Q23}
Q_{23}&=\int^{x_0+s_3^-\Delta{t}}_{x_0+s_2\Delta{t}}(\widebar{U}\p)(x,\Delta{t})dx
+\int^{\Delta{t}}_{0}\big((f(U)-f(\widetilde{U})-s_3^-\widebar{U})\p\big)(s_3^-t, t)dt\nonumber\\
&\quad-\int^{\Delta{t}}_{0}\big((f(U)-f(\widetilde{U})-s_2\widebar{U})\p\big)(s_2t+, t)dt+O(1)(\Delta t)^3(\Delta x)\|\phi\|_\infty,
\end{align}
and
\begin{align}
\label{Q25}
Q_{25}&=\int^{x_0+\D x}_{x_0+s_3^+\D t}(\widebar{U}\p)(x,\Delta{t})dx-
\int_{x_0}^{x_0+\Delta{x}}(\widebar{U}\p)(x,0)dx
+\int^{\Delta{t}}_{0}s_3^+(\widebar{U}\p)(s_3^+t,t)dt\nonumber\\
&\quad+\int^{\Delta{t}}_{0}\big((f(U)-f(\widetilde{U}))\p\big)(x, t)\Big|^{x=\D x}_{x=s_3^+t}dt+O(1)(\Delta t)^3(\Delta x)\|\phi\|_\infty.
\end{align}

Next, $Q_{24}$ is estimated as follows. In the region $D_4=\{(x,t)\mid\frac{x-x_0}{t}\in[s_3^-,s_3^+]\}$, $\widetilde{U}$ is the 3-rarefaction wave; therefore, by \eqref{aproxsol4}  the approximate solution $U$ satisfies
$$
U(x,t)=S(x_0+t\xi,t,\widetilde{U}(\xi))\widetilde{U}(\xi),
$$
where $\xi=\frac{x-x_0}{t}$. It leads to
\begin{eqnarray}
\label{diff2}
&&\widebar{U}_t-h(x_0+t\xi)g(x_0+t\xi,U) \nonumber\\
&&\hspace{1cm}=\partial_t\big((S-I_3)(x_0+t\xi,t,\widetilde{U}(\xi))\widetilde{U}(\xi)\big)
             -h(x_0+t\xi)g\big(x_0+t\xi,S(x_0+t\xi,t,\widetilde{U}(\xi))\widetilde{U}(\xi)\big) \nonumber\\
&&\hspace{1cm}=\partial_t(S-I_3)\widetilde{U}(\xi)-\frac{\xi}{t}\partial_{\xi}(S-I_3)\widetilde{U}(\xi)-\frac{\xi}{t}(S-I_3)\dot{\widetilde{U}}(\xi)
               -h(x_0+t\xi)g(x_0+t\xi,S\widetilde{U}) \nonumber\\
&&\hspace{1cm}=Q_{24}^1(\xi,t)+Q_{24}^2(\xi,t),
\end{eqnarray}
where $\dot{\widetilde{U}}:=\frac{d\widetilde{U}}{d\xi}$, and
\begin{align*}
Q_{24}^1(\xi,t)&\equiv\partial_t(S-I_3)\widetilde{U}(\xi)-h(x_0+t\xi)g(x_0+t\xi,S\widetilde{U}),\\
Q_{24}^2(\xi,t)&\equiv-\frac{\xi}{t}\Big(\partial_{\xi}(S-I_3)\widetilde{U}(\xi)+(S-I_3)\dot{\widetilde{U}}(\xi)\Big).
\end{align*}
According to \eqref{solver} and further calculation, we have
\begin{eqnarray}
\label{Q241}
&&Q_{24}^1(\xi,t)
=\frac{2\xi t}{x_0^2}\Big(\tilde{m},\tilde{\r}(\tilde{u}^2+2v_0^2),\tilde{m}(\widetilde{H}+2v_0^2)\Big)^T
   +O(1)(\D t)^2, \\
\label{Q242}
&&Q_{24}^2(\xi,t)
=\frac{\xi}{x_0}\left[\begin{array}{ccc}
0 & 2 & 0 \\
2(v_0^2-\tilde{u}^2) & 4\tilde{u} & 0 \\
(\g-1)\tilde{u}^3-2\tilde{u}\widetilde{H} & 2(\widetilde{H}-(\g-1)\tilde{u}^2+v_0^2) & 2\g\tilde{u}
\end{array}\right]\dot{\widetilde{U}} \nonumber \\
&&\hspace{2cm}+O(1)\D t,
\end{eqnarray}
where $v_0:=v(x_0)$.
Applying the integration by parts to $Q_{24}$ along with \eqref{diff2}, we obtain
\begin{align}
\label{Q24}
Q_{24}&=\int^{x_0+s_3^+\Delta{t}}_{x_0+s_3^-\Delta{t}}(\widebar{U}\p)(x,\Delta{t})dx
+\int^{\Delta{t}}_{0}\big((f(U)-f(\widetilde{U})-s_3^+\widebar{U})\p\big)(s_3^+t, t)dt\nonumber\\
&\quad-\int^{\Delta{t}}_{0}\big((f(U)-f(\widetilde{U})-s_3^-\widebar{U})\p\big)(s_3^-t, t)dt\nonumber\\
&\quad-\int_0^{\D t}\!\!\!\int_{s_3^-}^{s_3^+}\{Q_{24}^1(\xi,t)+Q_{24}^2(\xi,t)\}\p(t\xi,t)d\xi dt.
\end{align}
According to, \eqref{perturbres} and \eqref{Q21}-\eqref{Q24}
\begin{align}
\label{Q2}
Q_2&=\sum_{i=1}^5Q_{2i}=\int^{\Delta{x}}_{-\Delta{x}}(\widebar{U}\p)(x,t)\Big|_{t=0}^{t=\D t}dx
+\int^{\Delta{t}}_{0}\big((f(U)-f(\widetilde{U})-s_1\widebar{U})\p\big)(x, t)\Big|^{x=s_1t-}_{x=s_1t+}dt\nonumber\\
&\quad+\int^{\Delta{t}}_{0}\big((f(U)-f(\widetilde{U})-s_2\widebar{U})\p\big)(x, t)\Big|^{x=s_2t-}_{x=s_2t+}dt
+\int^{\Delta{t}}_{0}\big((f(U)-f(\widetilde{U}))\p\big)(x, t)\Big|_{x=-\Delta x}^{x=\Delta{x}}dt\nonumber\\
&\quad+O(1)\left((\Delta t)^3(\Delta x)+(\Delta{t})^3+(\Delta{t})^2(\underset{D}{\osc}\{\widetilde{U}\})\right)\|\phi\|_\infty.
\end{align}

We estimate the second and third terms on the right-hand side of \eqref{Q2}. Suppose that the state $\widetilde{U}_1$ is connected to the
state $\widetilde{U}_L=U_L=(\r_L,m_L,E_L)$ by 1-shock on the right and the $\widetilde{U}_2$ by the 2-contact discontinuity on the left. Then, the Rankine-Hugoniot condition \eqref{jumpcond} gives
\begin{equation}
\label{RHcond}
f(\widetilde{U}_1)-f(\widetilde{U}_L)=s_1(\widetilde{U}_1-\widetilde{U}_L),\quad
f(\widetilde{U}_1)-f(\widetilde{U}_2)=s_2(\widetilde{U}_1-\widetilde{U}_2).
\end{equation}
According to \eqref{aproxsol4} and \eqref{RHcond}, we obtain
\begin{equation}
\label{shockres}
\int_0^{\D t}\big((f(U)-f(\widetilde{U})-s_1\widebar{U})\p\big)(x,t)\Big|_{x=s_1t+}^{x=s_1t-}dt
=\Big(O(1)(\D t)^2(\underset{D}{\osc}\{\widetilde{U}\})+O(1)\D t(\underset{D}{\osc}\{\widetilde{U}\})^2\Big)\|\p\|_{\infty},
\end{equation}
because
\begin{eqnarray}
\label{solverestmate}
&&  \int_0^{\D t}\big((f(U)-f(\widetilde{U})-s_1\widebar{U})\p\big)(x,t)\Big|_{x=s_1t+}^{x=s_1t-}dt \nonumber\\
&=& \int_0^{\D t}\big(\big(f(S\widetilde{U}_L)-f(S\widetilde{U}_1)-S(f(\widetilde{U}_L)-f(\widetilde{U}_1))\big)\p\big)(s_1t,t)dt \nonumber\\
&=& \int_0^{\D t}\big(\big(f(S\widetilde{U}_L)-f(S\widetilde{U}_L+S\widetilde{U}_1-S\widetilde{U}_L)\big)\p\big)(s_1t,t)dt \nonumber\\
&&  \quad-\int_0^{\D t}\big(S\big(f(\widetilde{U}_L)-f(\widetilde{U}_L+\widetilde{U}_1-\widetilde{U}_L)\big)\p\big)(s_1t,t)dt
    +O(1)\D t(\underset{D}{\osc}\{\widetilde{U}\})^2\|\p\|_{\infty}\nonumber\\
&=& \int_0^{\D t}\big(\big(df(S\widetilde{U}_L)\cdot S\cdot(\widetilde{U}_1-\widetilde{U}_L)-S\cdot
    df(\widetilde{U}_L)\cdot(\widetilde{U}_1-\widetilde{U}_L)\big) \p\big)(s_1t,t)dt
    +O(1)\D t(\underset{D}{\osc}\{\widetilde{U}\})^2\|\p\|_{\infty}\nonumber\\
&=& \int_0^{\D t}\big(\big(df(S\widetilde{U}_L)\cdot S-S\cdot df(\widetilde{U}_L)\big)\cdot(\widetilde{U}_1-\widetilde{U}_L)\p\big)(s_1t,t)dt
    +O(1)\D t(\underset{D}{\osc}\{\widetilde{U}\})^2\|\p\|_{\infty}\nonumber\\
&=& \Big(O(1)(\D t)^2(\underset{D}{\osc}\{\widetilde{U}\})+O(1)\D t(\underset{D}{\osc}\{\widetilde{U}\})^2\Big)\|\p\|_{\infty}.
\end{eqnarray}
The aforementioned final equality holds because of the complex computation; therefore, we obtain:
\begin{eqnarray*}
df(S\widetilde{U}_L)\cdot S-S\cdot df(\widetilde{U}_L)
&=& df(S(x_0,t,\widetilde{U}_L)\widetilde{U}_L)\cdot S(x_0,t,\widetilde{U}_L)-S(x_0,t,\widetilde{U}_L)\cdot df(\widetilde{U}_L) \\
&=& \Big(hv^2t,-\frac{(\g-1)xhqt}{2\,\tilde{\r}_L},-\frac{(\g-1)^2h\tilde{u}_L^4t}{4}\Big)\Big|_{x=x_0}+O(1)(\D t)^2.
\end{eqnarray*}
Similarly, we have
\begin{equation}
\label{discontres}
\int_0^{\D t}\big((f(U)-f(\widetilde{U})-s_2\widebar{U})\p\big)(x,t)\Big|_{x=s_2t+}^{x=s_2t-}dt
=\Big(O(1)(\D t)^2(\underset{D}{\osc}\{\widetilde{U}\})+O(1)(\D t)(\underset{D}{\osc}\{\widetilde{U}\})^2\Big)\|\p\|_{\infty}.
\end{equation}
Therefore, by \eqref{shockres} and \eqref{discontres}, \eqref{Q2} can be rewritten as
\begin{align}
Q_2&=\int_{-\D x}^{\D x}(\widebar{U}\p)(x,t)\Big|^{t=\D t}_{t=0}dx
     +\int_0^{\D t}\big((f(U)-f(\widetilde{U}))\p\big)(x,t)\Big|_{x=-\D x}^{x=\D x}dt\nonumber\\
&\quad+O(1)\left((\Delta t)^3(\Delta x)+(\Delta{t})^3+(\Delta{t})^2(\underset{D}{\osc}\{\widetilde{U}\})\right)\|\phi\|_\infty.
\end{align}
Since $\widetilde{U}$ is the entropy solution for $\CMcal{R}_C(D)$, then according to the results of \cite{G,S2} and \eqref{RHcond}, we have
\begin{align}
\label{Q1}
\begin{split}
Q_1&=\iint_D\left\{\widetilde{U}\phi_t+f(\widetilde{U})\phi_x\right\}dxdt=\int_{-\D x}^{\D x}(\widetilde{U}\p)(x,t)\Big|_{t=0}^{t=\D t}dx
     +\int_0^{\D t}\big(f(\widetilde{U})\p\big)(x,t)\Big|_{x=-\D x}^{x=\D x}dt.
\end{split}
\end{align}
Following \eqref{res2}, \eqref{Q2} and \eqref{shockres}-\eqref{Q1}, we obtain \eqref{RPres}. The proof is complete.
\end{proof}

\section{Existence of global entropy solutions for IBVP \eqref{IBVP}}
\setcounter{equation}{0}

In this section, we establish the stability of the generalized Glimm scheme and consequently the compactness of the subsequences of approximate solutions $\{U_{\th,\D x}\}$ to \eqref{IBVP}. The stability, which is always the core of the Glimm method, is obtained through the modified wave interaction estimates, nonincreasing property of the Glimm functional, and uniform boundedness of the total variations of the perturbations in approximate solutions. We prove the global existence of entropy solutions to \eqref{IBVP} by demonstrating the consistency of the scheme and entropy inequalities for weak solutions at the end of this section.

\subsection{Generalized Glimm scheme for \eqref{IBVP}}

In this subsection, we introduce a nonstaggered generalized Glimm scheme for the initial boundary value problem \eqref{IBVP}. Let us discretize the domain $\Pi\equiv[x_B,\infty)\times[0,\infty)$ into
\begin{equation*}
x_k=x_B+k\Delta x,\quad t_n=n\Delta t,\quad k, n=0,1,2,\cdots,
\end{equation*}
where $\Delta x$ and $\Delta t$ are small positive constants satisfying the      CFL condition \eqref{CFL}. The $n$th time strip $T_n$ is denoted by
$$
T_{n}:=[x_B,\infty)\times[t_n,t_{n+1}),\quad n=0,1,2,\cdots.
$$
Suppose that the approximate solution $U_{\theta,\Delta{x}}(x,t)$ has been constructed in $T_n$; then, we choose a random number $\theta_n \in (-1, 1)$ and define the initial data $U_k^n\equiv(\rho_k^n,m_k^n,E_k^n)$ in $T_n$ by
$$
U_k^n:=U_{\theta,\Delta{x}}(x_{2k}+\theta_n\Delta{x},t_n^-),\ k=1,2,\cdots.
$$
To initiate the scheme at $n=0$, we set $t_0^-=0$. The points $\{(x_{2k}+\theta_n\Delta{x},t_n^-)\}_{n=0,k=1}^{\infty}$ are called the $mesh \; points$ of the scheme and the points $\{(x_B,t_n+\frac{\D t}{2})\}_{n=0}^{\infty}$ are the mesh points on the boundary $x=x_B$. Then, $U_{\theta,\Delta{x}}$ in $T_{n+1}$ is constructed by solving the set of Riemann problems $\{\CMcal{R}_G(x_k,t_n;g)\}_{k\in\NN}$ with initial data
\begin{align*}
U(x,t_n)=\left\{
\begin{array}{ll}
U^n_k,&\; \mbox{if } x_{2k-1}\leq x<x_{2k},\\
U^n_{k+1},&\;\mbox{if }  x_{2k}<x\leq x_{2k+1},
\end{array}\right.\ k=1,2,\cdots,
\end{align*}
and the boundary-Riemann problem $\CMcal{BR}_G(x_B,t_n;g)$ with initial-boundary data
\begin{alignat}{2}
\label{3.0.3}
\left\{\begin{array}{ll}
U(x,t_n)=U^n_1,&\mbox{if } x_0<x\leq x_1,\\
\r(x_B,t)=\r^n_B:=\r_B(t_{n}),  &\mbox{if } t_n\leq t< t_{n+1},\\
m(x_B,t)=m^n_B:=m_B(t_{n}),  &\mbox{if } t_n\leq t< t_{n+1}.
\end{array}\right.
\end{alignat}
Moreover, near the boundary $x=x_B$, the approximate solution for \eqref{3.0.3} satisfies the entropy condition $(\CMcal{E})$ in Section 2.
\begin{center}
\includegraphics[height = 4cm, width=10cm]{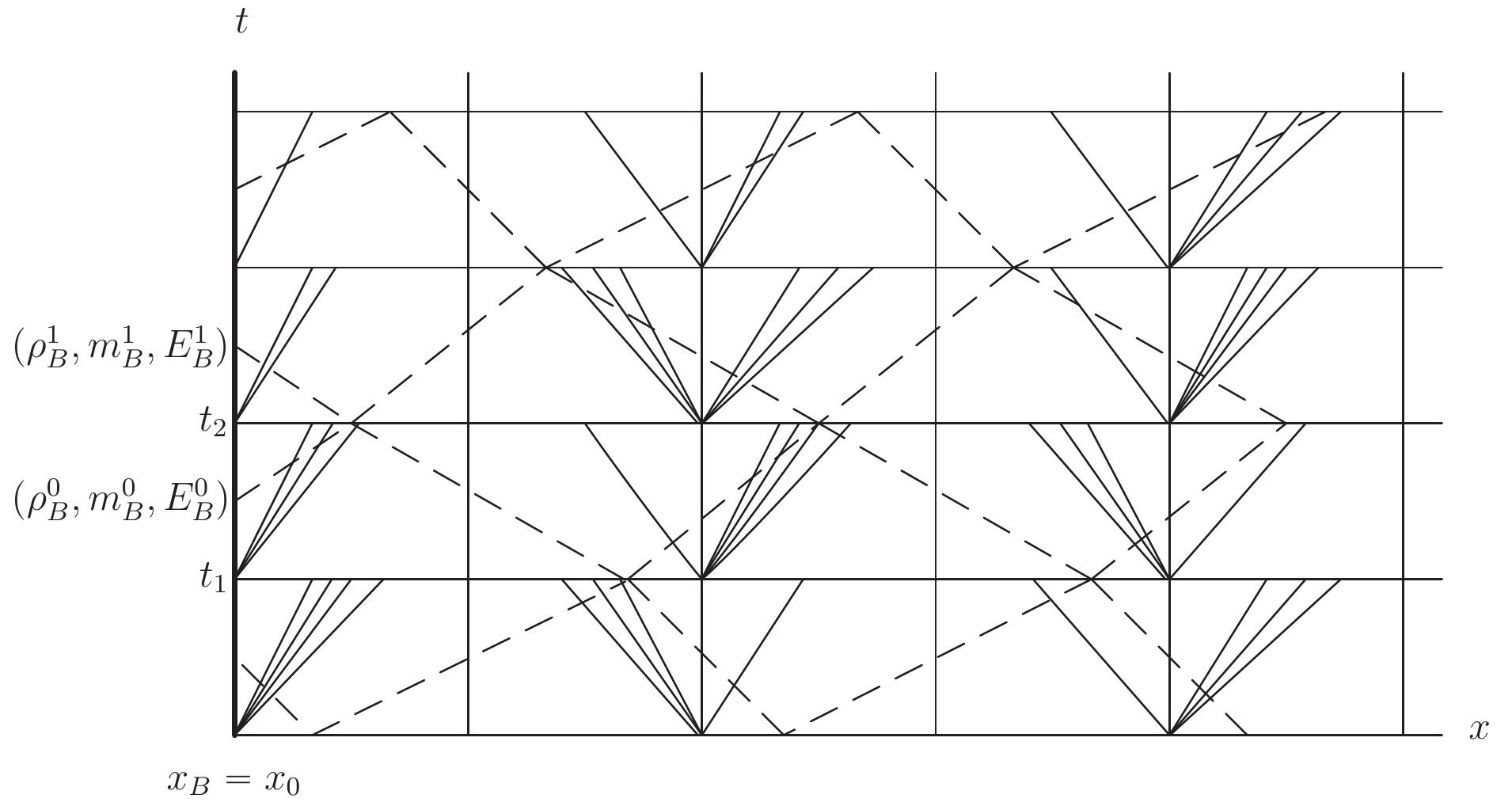}\\
Figure 3. Approximate solution for IBVP with mesh curves.\medskip
\end{center}
According to \eqref{aproxsol4}, the approximate solution $U_{\theta ,\Delta{x}}$ has an explicit representation
\begin{equation*}
U_{\theta ,\Delta{x}}(x,t)=S(x,t,\widetilde{U}_{\th,\D x}(x,t))\widetilde{U}_{\th,\D x}(x,t),\quad (x,t) \in T_{n},
\end{equation*}
where $\widetilde{U}_{\theta,\Delta{x}}$ consists of the weak solutions to the corresponding classical Riemann problems $\{\CMcal{R}_C(x_k,t_n)\}$ or boundary-Riemann problems $\CMcal{BR}_C(x_B,t_n)$ in $T_{n}$. The CFL condition \eqref{CFL} ensures that the elementary waves in each $T_{n}$ do not interact with each other before time $t=t_{n+1}$. Repeating this process, we construct the approximate solution $U_{\theta,\Delta x}$ of \eqref{IBVP}
in $\Pi$ by using the generalized Glimm scheme with a random sequence $\theta:= (\theta_0, \theta_1, \theta_2,\ldots)$, $\th_i\in(-1,1)$ for all $i$.

To obtain the desired estimates, it is favorable to consider curves comprising line segments joining mesh points rather then horizontal lines. Therefore, we define the mesh curves for the non-local Glimm functionals introduced in \cite{G}.
A mesh curve $J$ for \eqref{IBVP} is a piecewise linear curve that connects the mesh point $(x_{2k}+\theta_n\Delta x,t_n)$ on the left with $(x_{2k+2}+\theta_{n+1}\Delta x,t_{n+1})$ or $(x_{2k+2}+\theta_{n-1}\Delta x,t_{n-1})$ on its right, $k,n=0,1,2,\cdots$ together with the line segments joining the points $(x_B,t_n+\frac{\D t}{2})$ and $(x_B+\th_n\D x,t_n)$ and some portion of the boundary (see Figure 3 and \cite{CHS1}). Simultaneously, the mesh curve $J$ divides the domain $[x_0,\infty)\times[0,\infty)$ into $J^+$ and $J^-$ regions such that $J^-$ contains the line $[x_0,\infty)\times\{0\}$. We can partially order two mesh curves by saying $J_2>J_1$ (or $J_2$ is a {\it successor} of $J_1$) if every mesh point of $J_2$ is either on $J_1$ or contained in $J_1^+$. In particular, $J_2$ is an {\it immediate successor} of $J_1$ if
$J_2>J_1$ and all mesh points on $J_2$ except one are on $J_1$.
A diamond region is a closed region enclosed by a mesh curve and its immediate successor.

\subsection{Wave interaction estimates}

In this subsection, several types of nonlinear wave interactions will be described and the classical wave strengths of $\{{U}_{\th,\D x}\}$ in each time step will be estimated through wave interactions between the classical waves and the perturbations in the previous time step.

Connecting all the mesh points through the mesh curves, the domain $\Pi$ is decomposed as a union of the sets of diamond, triangular, and pentagonal regions. The wave interactions can be divided into the following three types:
\begin{enumerate}
\item [(I)] In each diamond region, the incoming generalized waves from adjacent Riemann problems interact with each other and emerge as the outgoing generalized waves of the Riemann problem in the next time step;
\item [(II)] In each triangular region, the incoming generalized waves from the Riemann problem at the boundary interact with each other and emerge as the outgoing generalized waves of the boundary-Riemann problem in the next time step;
\item [(III)] In each pentagonal region, two families of incoming generalized waves, one from the boundary-Riemann problem and the other from adjacent Riemann problem, interact with each other and emerge as the outgoing generalized waves of the Riemann problem in the next time step.
\end{enumerate}
In each diamond (or triangular and pentagonal) region, all the generalized waves comprise classical outgoing waves and perturbations. Therefore, the objective of wave interaction estimates is to estimate how the wave strengths of classical outgoing waves are influenced by the interaction or reflection of generalized incoming waves.

We start with the wave interaction estimates of type (I).
Suppose $(x,t)\in (x_B,\infty)\times[0,\infty)$ and let $\CMcal{R}_G(U_R,U_L;x,t)$ denote the generalized Riemann solution of $\CMcal{R}_G(x,t;g)$ connecting the left constant state $U_L$ with the right constant state $U_R$. Moreover let $\CMcal{R}_C(U_R,U_L;x,t)$ be the solution of the corresponding classical Riemann problem $\CMcal{R}_C(x,t)$. Then, the {\it classical wave strength} of $\CMcal{R}_G(U_R,U_L;x,t)$ is defined as the wave strength of $\CMcal{R}_C(U_R,U_L;x,t)$, which is expressed as
\begin{equation*}
\ve=\varepsilon(U_R,U_L;x,t)=(\ve_1,\ve_2,\ve_3).
\end{equation*}
In other words, the jump discontinuity $\{U_L,U_R\}$ is resolved into $U_L=\widetilde{U}_0$, $\widetilde{U}_1,\ \widetilde{U}_2$, and $\widetilde{U}_3=U_R$ such that $\widetilde{U}_{j}$ is connected to $\widetilde{U}_{j-1}$ on the right by a $j$-wave of strength $\varepsilon_j$. Note that $\CMcal{R}_C(U_R,U_L;x,t)$ is independent of the choice of $(x,t)$. We say that an $i$-wave and a $j$-wave approach if either $i>j$, or else $i = j$ and at least one wave is a shock. Given another $\CMcal{R}_G(U'_R,U'_L;x',t')$ with classical wave strength $\ve'=(\ve_1',\ve_2',\ve_3')$, the {\it wave interaction potential} associated with $\ve$, $\ve'$ is defined as
\begin{equation*}
D(\varepsilon,\varepsilon'):=\sum_{App}\{|\varepsilon_i\varepsilon_j'|:\varepsilon_i \mbox{ and } \varepsilon'_j \mbox{ approach}\}.
\end{equation*}
Assume that $J'$ is an immediate successor of $J$. Let $\Gamma_{k,n}$ denote the diamond region centered at $(x_{2k},t_{n})$ and enclosed by $J$ and $J'$. Four vertices of $\Gamma_{k,n}$, see Figure 4, are denoted by
\begin{align*}
   \begin{array}{ll}
     \mathpzc{N}=(x_{2k}+\theta_{n+1}\Delta x,t_{n+1}),& \mathpzc{E}=(x_{2k}+\theta_{n}\Delta x,t_{n}),\\
     \mathpzc{W}=(x_{2k+2}+\theta_{n}\Delta x,t_{n}), & \mathpzc{S}=(x_{2k}+\theta_{n-1}\Delta x, t_{n-1}),
   \end{array}
 \end{align*}
or
\begin{align*}
   \begin{array}{ll}
     \mathpzc{N}=(x_{2k}+\theta_{n+1}\Delta x,t_{n+1}),& \mathpzc{E}=(x_{2k-2}+\theta_{n}\Delta x,t_{n}),\\
     \mathpzc{W}=(x_{2k}+\theta_{n}\Delta x,t_{n}), & \mathpzc{S}=(x_{2k}+\theta_{n-1}\Delta x, t_{n-1}),
   \end{array}
 \end{align*}
Here $\{\theta_{n-1}, \theta_{n},\theta_{n+1} \}$ are random numbers in $(-1,1)$.

Define the matrix $R(U):=[R_1(U),R_2(U),R_3(U)]$, where $R_j$ is the right eigenvector of $df$ associated with the eigenvalue $\lambda_j$, and $R(U)$ is invertible in $\Omega$. Then, we have the following theorems on the wave interaction estimates.

\begin{thm}
\label{thm3.1}
Let $U_L,\ U_R$, and $U_M$ be constant states in some neighborhood contained in $\Omega$ with $U_L=\widetilde{U}_L+\widebar{U}_L$ and $U_R=\widetilde{U}_R+\widebar{U}_R$, where $\widetilde{U}_L(x_L,t_n):=(\widetilde{\rho}_L,\widetilde{m}_L,\widetilde{E}_L)^T,\ \widetilde{U}_R(x_R,t_n):=(\widetilde{\rho}_R,\widetilde{m}_R,\widetilde{E}_R)^T$, and $U_M(x_M,t_n):=(\rho_M,m_M,E_M)^T$
\begin{eqnarray}
\label{aproxperturb}
&&\begin{split}
&\widebar{U}_L=\big(S(x_L,\D t,\widetilde{U}_L)-I_3\big)\widetilde{U}_L=:(S_L-I_3)\widetilde{U}_L, \\
&\widebar{U}_R=\big(S(x_R,\D t,\widetilde{U}_R)-I_3\big)\widetilde{U}_R=:(S_R-I_3)\widetilde{U}_R,
\end{split} \\
&&x_L:=x_{2k-2}+\th_n\D x,\quad x_R:=x_{2k+2}+\th_n\D x, \nonumber\\
&&x_M=x_{2k}+\th_n\D x.\nonumber
\end{eqnarray}
Suppose that the classical wave strengths of the incoming generalized waves across the boundaries $\mathpzc{WS}$ and $\mathpzc{SE}$ of $\Gamma_{k,n}$ are
\begin{equation*}
\a(U_M,\widetilde{U}_L;x_L,t_{n-1})=(\a_1,\a_2,\a_3) \quad\text{and}\quad
\beta(\widetilde{U}_R,U_M;x_R,t_{n-1})=(\beta_1,\beta_2,\b_3),
\end{equation*}
\begin{center}
\includegraphics[height=4cm, width=7.5cm]{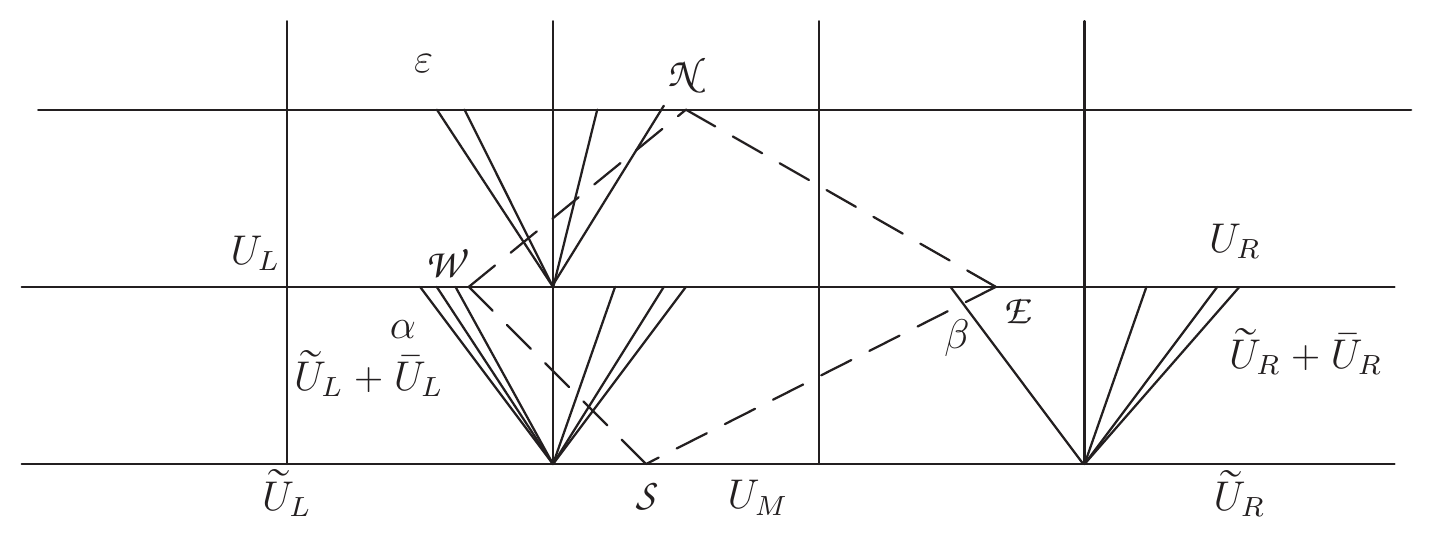}\quad
\includegraphics[height=4cm, width=7.5cm]{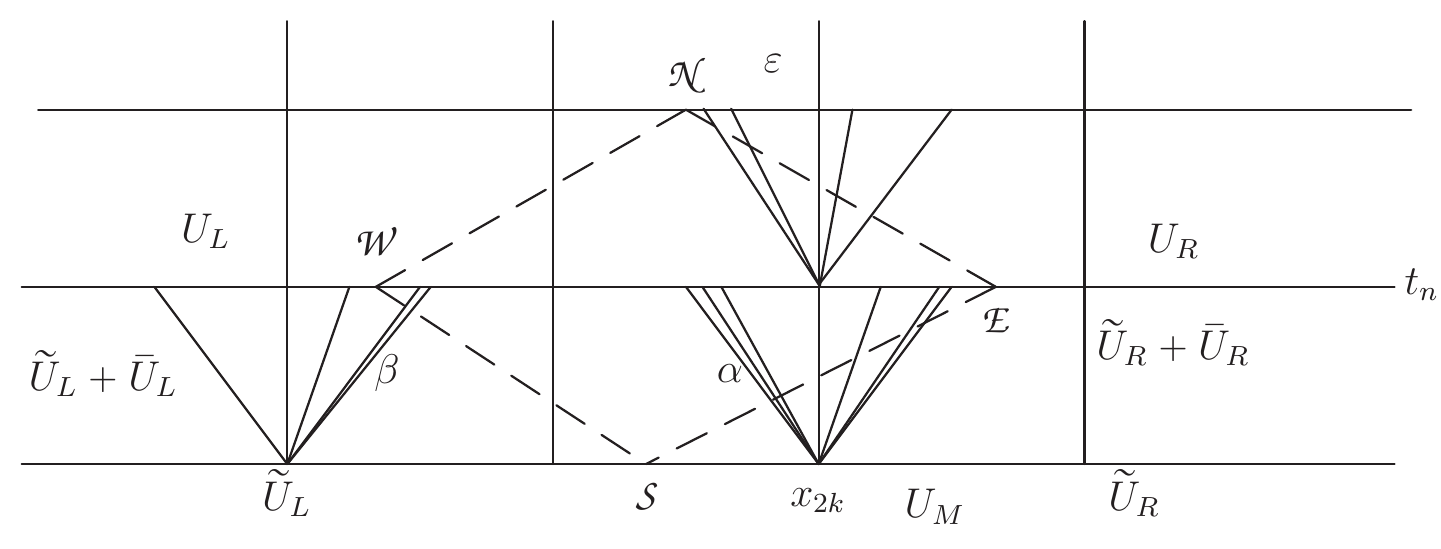}\\
Figure 4. Classical wave strengths in the diamond region $\Gamma_{k,n}$.\medskip
\end{center}
respectively, and the classical wave strength of the outgoing generalized waves across the boundary $\mathpzc{WNE}$ {\rm(}see Figure 4{\rm)} is
\begin{equation*}
\varepsilon(U_R,U_L;x_M,t_n)=(\varepsilon_1,\varepsilon_2,\ve_3),
\end{equation*}
Then there exist constants $C'$ and $C''_{k,n}$ such that
\begin{equation}
\label{3.6}
|\varepsilon|\le(1-\z\D t)(|\a|+|\beta|)+C'D(\a,\beta)+C''_{k,n}\D t\D x+O(1)(\D t)^3,\ \text{as}\ |\a|+|\beta|\rightarrow 0,
\end{equation}
where $\z=\frac{(3-\g)|u_M|}{x_M}$ and
$$
C''_{k,n}:=\sqrt{3}\Big(\frac{\r_Mu_M}{x_M^2}+\frac{2\r_Mv_B^2}{x_M^2c_M}+\frac{(\g-1)|q'_M|}{c_M^2}\Big).
$$
In particular, \eqref{3.6} is reduced to the classical wave interaction estimate in {\rm \cite{G}} when the system \eqref{3x3system} is without the source term.
\end{thm}
\begin{proof}
According to the results of \cite{S2}, we have
\begin{eqnarray}
\label{3.7}
&&U_R-U_M=\widebar{U}_R+\sum_{j=1}^3\beta_jR_j+ \sum_{j\leqslant i}\beta_j\beta_i\big(R_j\cdot\nabla
          R_i\big)\left(1-\frac{\delta_{ij}}{2}\right)+O(1)|\beta|^3, \\
\label{3.8}
&&U_L-U_M=\widebar{U}_L+\sum_{j=1}^3(-\a_j)R_j+\sum_{j\geqslant i}\a_j\a_i\big(R_j\cdot\nabla R_i\big)\left(1-\frac{\delta_{ij}}{2}\right)+O(1)|\a|^3,
\end{eqnarray}
where $R_j$ is the right eigenvector of $df$ associated with the eigenvalue $\lambda_j$, and $\delta_{ij}$ is the Kronecker delta. In addition, the coefficients in \eqref{3.7} and \eqref{3.8} are all evaluated at $U_M$. Similarly,
\begin{align}
\label{3.9}
U_R-U_L&=\sum_{i=1}^3\varepsilon_iR_i(U_L)+\sum_{j\leqslant i}\varepsilon_j\varepsilon_i\big(R_j\cdot\nabla
         R_i\big)(U_L)\left(1-\frac{\delta_{ij}}{2}\right)+O(1)|\varepsilon|^3\nonumber\\
&=\sum_{i=1}^3\varepsilon_iR_i(\widetilde{U}_L)+\sum_{j\leqslant i}\varepsilon_j\varepsilon_i\big(R_j\cdot\nabla
  R_i\big)(\widetilde{U}_L)\left(1-\frac{\delta_{ij}}{2}\right)\nonumber\\
&\quad+\left(\sum_{i=1}^3\varepsilon_i\nabla R_i(\widetilde{U}_L)\right)\widebar{U}_L+\CMcal{C}(\Delta x,|\varepsilon|),
  R_i\big)\left(1-\frac{\delta_{ij}}{2}\right)+O(|\varepsilon|^3)+(\widebar{U}_R-\widebar{U}_L),
\end{align}
where $\CMcal{C}(\Delta x,|\varepsilon|)$ denotes the cubic terms of $\Delta x$ and $|\varepsilon|$. According to \eqref{3.7}-\eqref{3.9}, $\varepsilon=0$ when $\a=\beta=0$ and $\widebar{U}_R-\widebar{U}_L=0$. By \eqref{2.13.1},
$$
\varepsilon=O(1)(\Delta x+|\a|+|\beta|).
$$
According to the Taylor expansion of $R_i$ around $U_M$,
\begin{equation}
\label{3.10}
R_i(\widetilde{U}_L)=R_i(U_M)-\sum_{j=1}^3\a_j(R_j\cdot\nabla R_i)(U_M)+O(1)|\a|^2.
\end{equation}
According to \eqref{3.9} and \eqref{3.10} together with $\varepsilon=O(1)(\Delta x+|\a|+|\beta|)$, the difference of  $U_R$ and $U_L$ is expressed as
\begin{align}
\label{3.11}
U_R-U_L
&=\sum_{i=1}^3\varepsilon_iR_i(U_M)+\sum_{j\leqslant i}\varepsilon_j\varepsilon_i(R_j\cdot\nabla
  R_i)(U_M)\left(1-\frac{\delta_{ij}}{2}\right)\nonumber\\
&\quad-\sum_{\scriptstyle i,j}\varepsilon_i\a_j(R_j\cdot\nabla
  R_i)(U_M)+\left(\sum_{i=1}^3(\a_i+\beta_i)\nabla R_i(U_M)\right)\widebar{U}_L\nonumber\\
&\quad+\CMcal{C}(\Delta x,|\a|+|\beta|).
\end{align}
According to \eqref{2.13.1} and \eqref{aproxperturb}, $\widebar{U}_L=O(1)(\Delta t)=O(1)(\Delta x)$ and
\begin{align}
\label{3.14}
\widebar{U}_R-\widebar{U}_L
&=(S_R-I_3)\widetilde{U}_R-(S_L-I_3)\widetilde{U}_L \nonumber\\
&=(S_M-I_3)(\widetilde{U}_R-\widetilde{U}_L)+(S_R-S_M)\widetilde{U}_R+(S_M-S_L)\widetilde{U}_L.
\end{align}
By comparing \eqref{3.11} with \eqref{3.7}, and \eqref{3.8}, and using \eqref{3.14}, we obtain
\begin{align}
\label{3.17}
\varepsilon^T
&=(\a+\beta)^T+\sum_{j<i}\a_i\beta_jL_{ij}(U_M)-R^{-1}(U_M)\left(\sum_{i=1}^3(\a_i+\beta_i)\nabla R_i(U_M)\right)(S_L-I_3)\widetilde{U}_L\nonumber\\
&\quad+R^{-1}(U_M)(S_M-I_3)(\widetilde{U}_R-\widetilde{U}_L)+R^{-1}(U_M)\big((S_R-S_M)\widetilde{U}_R+(S_M-S_L)\widetilde{U}_L\big)\nonumber\\
&\quad+\CMcal{C}(\Delta x,|\a|+|\beta|),
\end{align}
where $R=[R_1,R_2,R_3]$ and $L_{ij}:=R^{-1}(R_i\cdot \nabla R_j-R_j\cdot \nabla R_i)$. To estimate \eqref{3.17}, we need to evaluate all terms at the state $U_M$. After a complex calculation and using the Taylor expansion with respect to $\D t$, the term $(S_L-I_3)\widetilde{U}_L$ is estimated as follows:
\begin{align}
\label{3.17-1}
(S_L-I_3)\widetilde{U}_L
&=(S_M-I_3)U_M+(S_M-I_3)(\widetilde{U}_L-U_M)+(S_L-S_M)\widetilde{U}_L \nonumber \\
&=(S_M-I_3)U_M+O(1)\osc\{U\}\D t+O(1)\D t\D x,
\end{align}
where $S_M:=S(x_M,\D t,U_M)$.

Next, we estimate $(S_R-S_M)\widetilde{U}_R+(S_M-S_L)\widetilde{U}_L$ in \eqref{3.17}. Define $\mathscr{F}(x,\D t,\widetilde{U}):=S(x,\D t,\widetilde{U})\widetilde{U}$; then,
\begin{align}
\label{3.17-2}
&(S_R-S_M)\widetilde{U}_R+(S_M-S_L)\widetilde{U}_L \nonumber\\
&=\mathscr{F}(x_R,\D t,\widetilde{U}_R)-\mathscr{F}(x_L,\D t,\widetilde{U}_L)-S_M(\widetilde{U}_R-\widetilde{U}_L) \nonumber\\
&=\mathscr{F}_x(x_M,\D t,U_M)\D x+(\mathscr{F}_U(x_M,\D t,U_M)-S_M)(\widetilde{U}_R-\widetilde{U}_L) \nonumber\\
&=W(U_M)\D t\D x+\Psi(U_M)(\widetilde{U}_R-\widetilde{U}_L)+\CMcal{C}(\Delta x,|\a|+|\beta|),
\end{align}
where
\begin{align*}
&W(U_M)=\frac{h_M^2}{2}\Big(\r_Mu_M,\r_M(u_M^2+2v_M^2),\r_Mu_M(H_M+2v_M^2)+\frac{x_M^2q_M'}{2}\Big)^T, \\
&\Psi(U_M)=\mathscr{F}_U(x_M,\D t,U_M)-S_M 
\end{align*}
with $h_M:=-\frac{2}{x_M},H_M:=H(U_M)$, and $q_M':=q'(x_M)$. According to \eqref{3.7} and \eqref{3.8}, the difference between $\widetilde{U}_R$ and $\widetilde{U}_L$ is
\begin{equation}
\label{3.19}
\widetilde{U}_R-\widetilde{U}_L=R(U_M)(\a+\beta)^T+O(1)(|\a_i||\a_j|+|\beta_i||\beta_j|).
\end{equation}
Applying the Taylor expansion of \eqref{3.17} at $U_M$ along with \eqref{3.19}, and using \eqref{3.17-1} and \eqref{3.19}, we have
\begin{align}
\label{3.18}
\varepsilon^T
&= \big(I_3-\widebar{D}(U_M)+R^{-1}(S_M-I_3)R\big)(\a+\beta)^T+\sum_{j<i}\a_i\beta_jL_{ij}(U_M) \nonumber\\
&\quad+R^{-1}(U_M)W(U_M)\D t\D x+R^{-1}(U_M)\Psi(U_M)(\widetilde{U}_R-\widetilde{U}_L)+\CMcal{C}(\Delta x,|\a|+|\beta|),
\end{align}
where
\begin{align*}
\widebar{D}(U_M)
:=R^{-1}(U_M)\big(\nabla R_1(U_M)(S_M-I_3)\widetilde{U}_M,\nabla R_2(U_M)(S_M-I_3)\widetilde{U}_M,\nabla R_3(U_M)(S_M-I_3)\widetilde{U}_M\big).
\end{align*}
Define
\begin{align*}
\Phi(U_M):=I_3-\widebar{D}(U_M)-R^{-1}(U_M)(S_M-I_3)R(U_M).
\end{align*}
By a direct calculation of eigenvalues $\nu_1,\nu_2,\text{and}\ \nu_3$ of $\Psi$ and $\m_1,\m_2,\text{and}\ \m_3$ of $\Phi$ evaluated at $x=x_M$, and considering $h_M=-\frac{2}{x_M}$, we obtain
\begin{equation*}
\nu_1=\nu_2=0,\qquad\nu_3=\frac{4v_M^2}{x_M^2}(\D t)^2+O(1)(\D t)^3,
\end{equation*}
\begin{equation*}
\begin{split}
\m_1&=1-\frac{2u_M}{x_M}\D t+O(1)(\D t)^2, \\
\m_2&=1-\Big(\frac{(3-\g)u_M}{x_M}+\frac{\g(\g-1)q_M}{2\r_Mc_M^2}\Big)\D t+O(1)(\D t)^2, \\
\m_3&=1-\Big(\frac{2u_M}{x_M}+\frac{\g(\g-1)q_M}{\r_Mc_M^2}\Big)\D t+O(1)(\D t)^2.
\end{split}
\end{equation*}
We obtain that $0<\m_i < 1,\ i=1,2,3$ when $\Delta t>0$ is sufficiently small. Therefore, there exist non-singular matrices $Q_{\D t}(U_M),\text{and}\ T_{\Delta t}(U_M)$ such that
\begin{align*}
\Psi(U_M)=Q_{\Delta t}\cdot\diag[\nu_1,\nu_2,\nu_3]\cdot Q_{\Delta t}^{-1}, \\
\Phi(U_M)=T_{\Delta t}\cdot\diag[\m_1,\m_2,\m_3]\cdot T_{\Delta t}^{-1}.
\end{align*}
Since $1<\g<\frac{5}{3}$, if $\z=\frac{(3-\g)|u_M|}{x_M}$ is selected, then,
\begin{equation}
\label{3.28}
\begin{split}
|\Psi(U_M)(\widetilde{U}_R-\widetilde{U}_L)|&\le|Q_{\D t}||\diag[\nu_1,\nu_2,\nu_3]||Q_{\D t}^{-1}||R||\a+\b|=\CMcal{C}(\Delta x,|\a|+|\beta|), \\
\left|\Phi(U_M)\right|&\le|T_{\Delta t}||\diag[\m_1,\m_2,\m_3]||T_{\Delta t}^{-1}|\le\max\limits_{1\le i\le3}\m_i\leq 1-\z\D t.
\end{split}
\end{equation}
Finally, by \eqref{3.18} and \eqref{3.28}, we obtain \eqref{3.6}. The proof is complete.
\end{proof}

For the cases (II) and (III), we construct the approximate solution for boundary-Riemann problem by \eqref{aproxsol4}. Due to the construction of the approximate solutions to the boundary-Riemann problems, the approximate solutions do not match the boundary conditions. We need to understand how the
the errors \eqref{2.14} on the boundary affect the interaction of waves. Therefore, the wave interaction near the boundary is more complicated than the wave interaction in case (I). On the other hand, to estimate the wave interaction near the boundary, the exact direction of wave for each characteristic field must be known. Because of the positivity of the initial and boundary velocities, we can prove that the velocity is globally positive, as shown in Section 3.3.

Let us denote the boundary data at $n$th time strip by
$$
\widehat{U}_B^n:=(\r(x_B,t_n),m(x_B,t_n),0)^T=(\r_B^n,m_B^n,0)^T
$$
and let $U_B^n=\widetilde{U}_B^n+\widebar{U}_B^n$ be the solution of the generalized boundary-Riemann problem $\CMcal{BR}_G(x_B,t_n)$. We define the strength of the 0-wave on the boundary as
$$
\a_0:=|R^{-1}(\widetilde{U}_B^n)\cdot(U_B^n-\widehat{U}_B^n)|.
$$
According to \cite{G1}, the following theorem can be proved using the generalized version of Goodman's wave interaction estimates near the boundary:

\begin{thm} {\rm(Boundary interaction estimate)}
Let $\widehat{U}_B^n$ and $U_B^n$ as defined previously, and $U_B^{n+1}:=\widetilde{U}_B^{n+1}+\widebar{U}_B^{n+1}$, where $U_M^n,\ U_R^n,\text{and}\  U_R^{n+1}$
 are defined as in Theorem 3.1., {\rm(}see Figure 5{\rm)}.
\begin{center}
\includegraphics[scale=0.8]{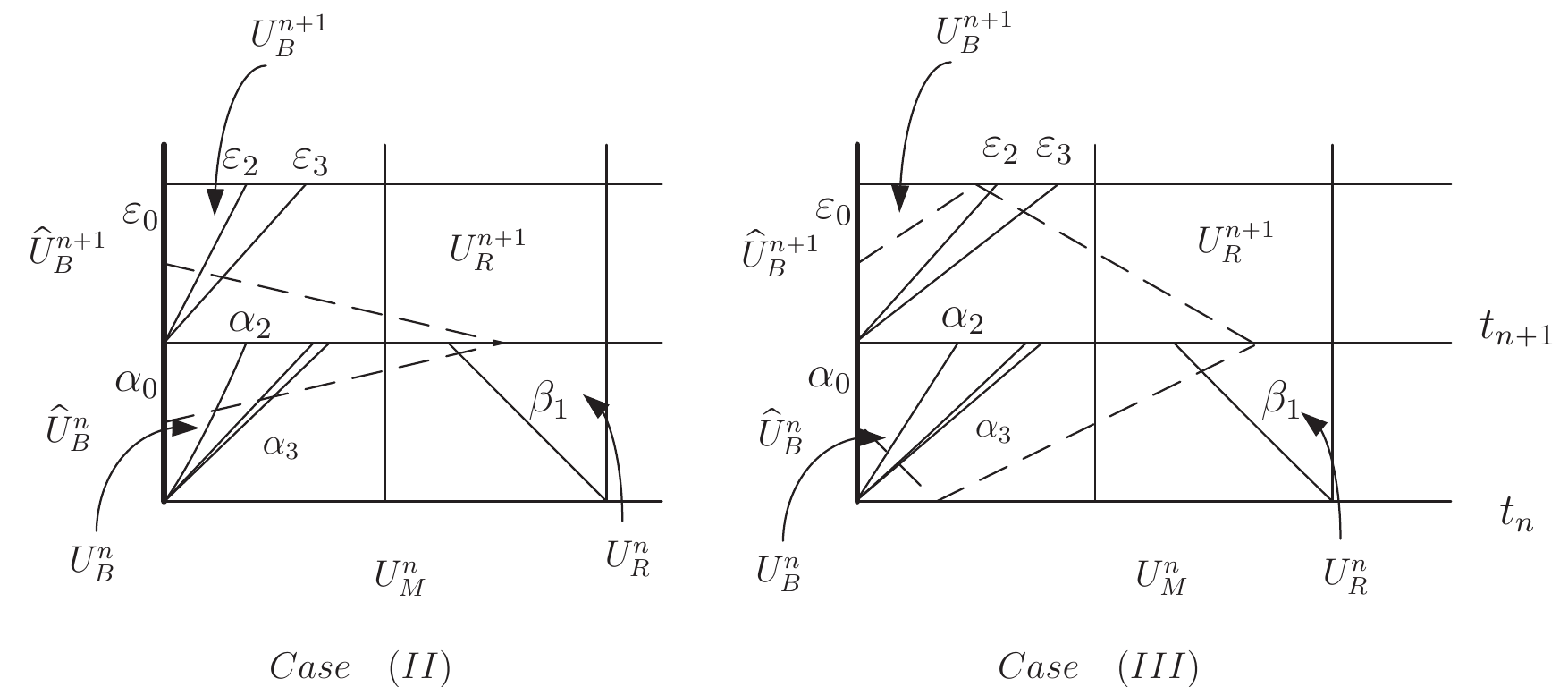}\\
Figure 5. Wave strengths in the region near the boundary $x=x_B$.\medskip
\end{center}
Moreover, let
$$
(\widehat{U}_B^n,U_M^n):=[(\widehat{U}_B^n,U_B^n,U^n_Z,U_M^n)/(\a_0,\a_2,\a_3)],\
(\widehat{U}_B^{n+1},U_R^{n+1}):=[(\widehat{U}_B^{n+1},{U}_B^{n+1},U^{n+1}_Z,U_R^{n+1})/(\varepsilon_0,\varepsilon_2,\ve_3)]
$$
represent the solutions of $\CMcal{BR}_G(\widehat{U}_B^n,U_M^n;x_B,t_n)$ and $\CMcal{BR}_G(\widehat{U}_B^{n+1},U_R^{n+1};x_B,t_{n+1})$, respectively. Assume that $(U_M^n,U_R^n):=[(U_M^n,U_R^n)/(\beta_1)]$ is the 1-wave of $\CMcal{R}_G(x_2,t_n)$ right next to $(\widehat{U}_B^n,U_M^n)$ on the $n$th time strip {\rm(}see Figure. 5{\rm)}. Then there exists a constant $C$ such that
\begin{align}
\label{4.26b}
|\varepsilon| &\leq |\a+\beta_1\mathbf{1}| +
C\Big(\sum_{App}|\a_i||\beta_1|+|\beta_1|+|\rho_B^{n+1}-\rho_B^n|+|m_B^{n+1}-m_B^n|\Big),
\end{align}
where $\mathbf{1}=(1,1,1)$.
\end{thm}
\begin{proof}
The boundary interaction estimate is calculated using the method reported in \cite{CHS2,G1}. The estimation is constructed using the following steps.
(1) Decompose the wave interaction into two parts, the transmission part and the reflection part, and evaluate the interacted wave strength of these parts. (2) Estimate the effect of the Riemann solver \eqref{solver} on the wave strength. (3) Combine the estimation of (1) and (2) to complete the proof.
Step (1): We divide the interaction of the waves into two parts, the transmission part and the reflection part. The incoming wave $\b_1$ interacts with $\a_3,\ \a_2$, and $\a_0$ in order; it generates one penetrative wave $\b_1'$ through the boundary, see Figure 6, and two reflected waves $\a_2'$ and $\a_3'$ from the boundary, see Figure 7. For the transmission part, the wave strength $\b_1'$ and states $U_1'$ and $U_R'$ can be determined through the interaction estimates of waves $\beta_1,\ \a_0,\ \a_2$, and $\a_3$. Furthermore, according to Lemma 4.2 (a) in \cite{CHS1} and the triangle inequality, we obtain
\begin{eqnarray}
\label{4.30b}
|U'_R-U_R^n|\leq C(|\a_0\beta_1|+|\a_2\beta_1|+|\a_3\b_1|).
\end{eqnarray}
It is evident that
\begin{eqnarray}
\label{4.31b}
|U'_1-\widehat{U}^n_B|\leq C_1|\beta_1'|\le C_2|\b_1|,
\end{eqnarray}
for some constants $C_1,C_2$.
\begin{center}
\includegraphics[height = 3.5cm, width=14cm]{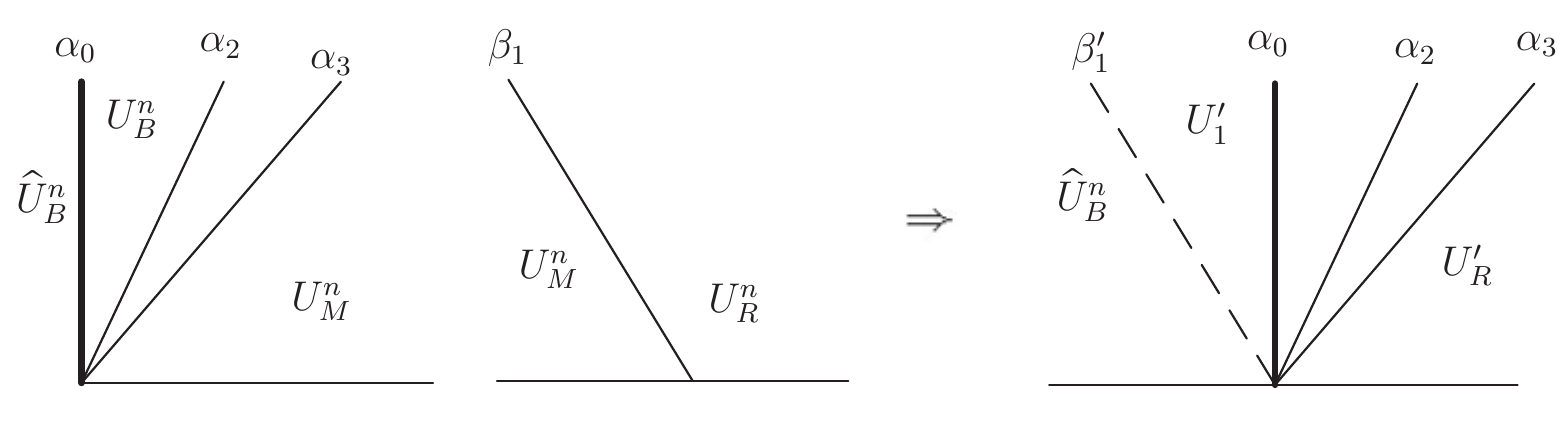}\\
Figure. 6: Interaction of waves for the transmission part $\ve=\a+\b+$ higher order term.\medskip
\end{center}
\begin{center}
\includegraphics[height = 3.5cm, width=13.5cm]{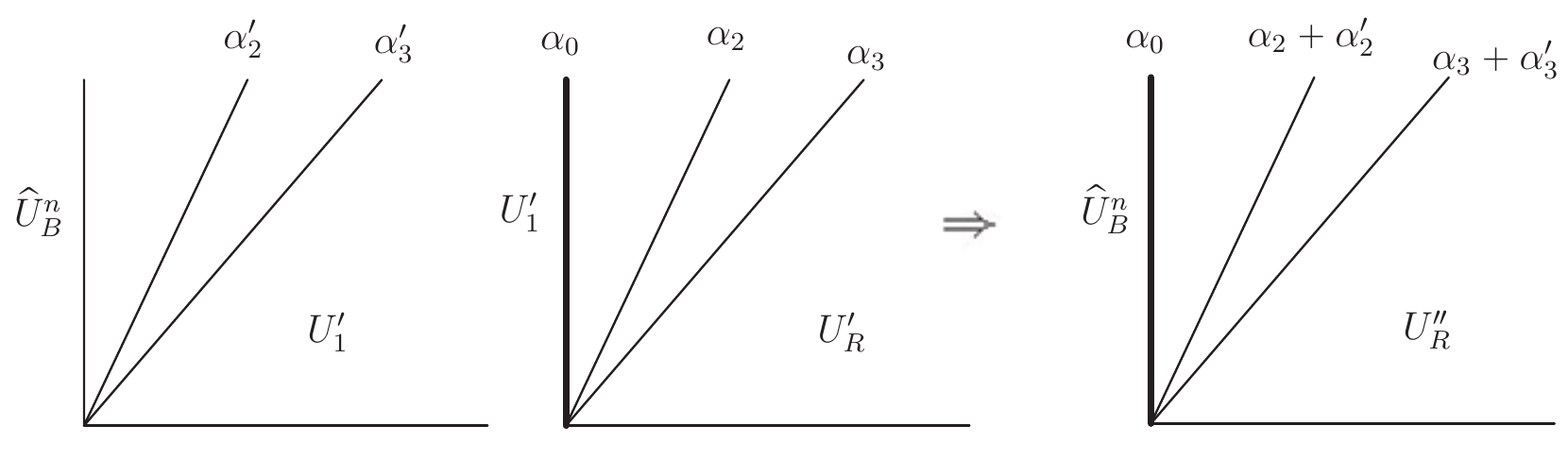}\\
Figure 7: Interaction of waves for the reflection part $\ve=\a+\a'+$ higher order term.\medskip
\end{center}
For the reflection part, the wave strengths $\a_2'$ and $\a_3'$ can be determined through the generalized boundary-Riemann problem $\CMcal{BR}_G(\widehat{U}_B^n;U_1')$. Therefore, we estimate the interaction of $\a_2'$ and $\a_3'$ and the waves $\a_2$ and $\a_3$. For any state $U=(\r,m,E)^T$, let us define
$$
|U|_{1,2}:=|\r|+|m|.
$$
According to \eqref{4.31b}, $\a_2'$ and $\a_3'$ satisfy
\begin{eqnarray}
\label{4.32b}
|(\a_2',\a_3')|\leq C|U_1'-\widehat{U}_B^n|_{1,2}\leq C|\beta_1|.
\end{eqnarray}
Therefore, according to Lemma 4.2 (a) and (b) in \cite{CHS1} and the triangle inequality, we obtain
\begin{align}
\label{4.33b}
|U''_R-U'_R|&\leq C(|\a_0\a_2'|+|\a_2\a_2'|+|\a_3\a_2'|+|\a_0\a_3'|+|\a_2\a_3'|+|\a_3\a_3'|)\leq C|(\a_2',\a_3')|,
\end{align}
where $U_R''$ is connected to $\widehat{U}_B'$ on the right by waves $\a_0,\ \a_2'+\a_2$, and $\a_3'+\a_3$. Wave strengths $\a_0$ and $\a_2$ are bounded; using this and \eqref{4.32b} and \eqref{4.33b}, we have
\begin{align}
\label{4.34b}
|U''_R-U'_R|\leq C|\beta_1|.
\end{align}
Let $(\widehat{U}^n_B, U^{n+1}_R):=[(\widehat{U}^n_B;U^{n+1}_R)/(\k_0,\k_2,\k_3)]$ be the approximate solution of $\mathcal{BR}_G(\widehat{U}^n_B;U^{n+1}_R)$, for some $\k_i,\ i=0,2,3$. According to the result in Section 2, there exists a smooth function $\Theta=(\Theta_0,\Theta_2,\Theta_3)$ connecting two constant states such that $(\varepsilon_0,\varepsilon_2,\ve_3)=\Theta(U^{n+1}_R;\widehat{U}^{n+1}_B)$ and $(\k_0,\k_2,\k_3)=\Theta(U^{n+1}_R;\widehat{U}^{n}_B)$. By \eqref{aproxsol4} and $|U^{n+1}_R-\widehat{U}^{n}_B|=O(1)|(\k_0,\k_2,\k_3)|$, we obtain
\begin{align}
\label{4.27d}
&\varepsilon_j-\k_j\nonumber\\
&= \int^1_0 \frac{d}{d\xi}\Theta_j(\widehat{U}^{n+1}_B+\xi(U^{n+1}_R-\widehat{U}^{n+1}_B);\widehat{U}^{n+1}_B) d \xi
   -\int^1_0 \frac{d}{d\xi}\Theta_j(\widehat{U}^{n}_B+\xi(U^{n+1}_R-\widehat{U}^{n}_B);\widehat{U}^{n}_B) d \xi\nonumber\\
&= \int^1_0 \{d\Theta_j(\widehat{U}^{n+1}_B+\xi(U^{n+1}_R-\widehat{U}^{n+1}_B);\widehat{U}^{n+1}_B)
   -d\Theta_j(\widehat{U}^{n}_B+\xi(U^{n+1}_R-\widehat{U}^{n}_B);\widehat{U}^{n}_B)\}\cdot(U^{n+1}_R-\widehat{U}^{n}_B)d\xi\nonumber\\
&\quad+\int_0^1d\Theta_j(\widehat{U}^{n+1}_B+\xi(U^{n+1}_R-\widehat{U}^{n+1}_B);\widehat{U}^{n+1}_B)
  \cdot(\widehat{U}^{n}_B-\widehat{U}^{n+1}_B) d \xi,\ j=0,2,3.
\end{align}
According to \eqref{4.27d},
\begin{align}
\label{4.28b}
|\varepsilon_j-\k_j|
&\le O(1)|(\k_0,\k_2,\k_3)|\big(|\rho^{n+1}_B-\rho^{n}_B|+|m^{n+1}_B-m^{n}_B|\big)+O(1)(|\rho^{n+1}_B-\rho^{n}_B|+|m^{n+1}_B-m^{n}_B|)\nonumber \\
&\leq C(|\rho^{n+1}_B-\rho^{n}_B|+|m^{n+1}_B-m^{n}_B|), \quad j=0,2,3.
\end{align}
Step (2): Let us denote $S_B:=S(x_B,t_n,\widehat{U}_B^n)=S(x_B,t_n,\widetilde{U}_B^n)$ and $S_R:=S(x_2,t_n,\widetilde{U}_R^n)$, where $S$ is in \eqref{aproxsol4}. According to \eqref{2.13.1}, \eqref{4.30b}, and \eqref{4.34b}, we obtain
\begin{align}
\label{4.35b}
&|\k^T-S(x_B,t_n,\widehat{U}_B^n)(\a+\a')^T|\nonumber\\
&\qquad\leq C\big|R^{-1}(\widetilde{U}_B^n)U_R^{n+1}-S_BR^{-1}(\widetilde{U}_B^n)U_R''\big|\nonumber\\
&\qquad\leq C(|U_R^{n+1}-S_RU_R^n|+|S_B||U_R-U_R'|+|S_B||U_R'-U_R''|)\nonumber\\
&\qquad\leq C(e^{h_B\tilde{u}_B^n\D t}\cosh(h_Bv_B\D t)|U_R-U_R'|+e^{h_B\tilde{u}_B^n\D t}\cosh(h_Bv_B\D t)|U_R'-U_R''|)\nonumber\\
&\qquad\leq Ce^{h_B\tilde{u}_B^n\D t}\cosh(h_Bv_B\D t)(|\a_0\beta_1|+|\a_2\beta_1|+|\a_3\b_1|+|\beta_1|),
\end{align}
where $\a_0':=0$, $h_B:=h(x_B)=-\frac{2}{x_B}$ and $v_B=v(x_B)$ as in \eqref{grav}. Finally, by \eqref{4.32b} and \eqref{4.35b}, we have
\begin{align}
\label{4.36b}
|\k^T-S(x_B,t_n,\widehat{U}_B^n)(\a+\b_1\mathbf{1})^T|
&\leq|\k-S(x_B,t_n,\widehat{U}_B^n)(\a+\a')^T|+e^{h_B\tilde{u}_B^n\D t}\cosh(h_Bv_B\D t)(|\a'|+|\beta_1|)\nonumber\\
&\leq Ce^{h_B\tilde{u}_B^n\D t}\cosh(h_Bv_B\D t)(|\a_0\beta_1|+|\a_2\beta_1|+|\beta_1|).
\end{align}
Step (3): Finally, according to \eqref{4.28b} and \eqref{4.36b},
\begin{align*}
&|\ve|\le|\ve-\k|+|\k^T-S(x_B,t_n,\widehat{U}_B^n)(\a+\b_1\mathbf{1})^T|+e^{h_B\tilde{u}_B^n\D t}\cosh(h_Bv_B\D t)|\a+\b_1|\nonumber\\
&\quad\leq |\a+\beta_1\mathbf{1}|+C\Big(\sum_{App}|\a_i||\beta_1|+|\beta_1|+|\rho_B^{n+1}-\rho_B^n|+|m_B^{n+1}-m_B^n|\Big).
\end{align*}
We complete the proof of the theorem.
\end{proof}

\noindent If $\beta_1$ is not the incoming wave of the boundary triangle region, then \eqref{4.26b} is reduced by
\begin{align*}
|\varepsilon| &\leq |\a|+C\big(|\rho_B^{n+1}-\rho_B^n|+|m_B^{n+1}-m_B^n|\big).
\end{align*}

\subsection{The stability of the generalized Glimm scheme}

In this subsection, we prove the nonincreasing of the Glimm functional and provide the estimate of the total variations of the perturbations, which lead to the compactness of subsequences of the approximate solutions for \eqref{IBVP}.
Let $U_{\theta, \Delta x}$ denote the approximate solution for \eqref{IBVP} by the generalized Glimm scheme described in Section
3.1; $U_{\theta, \Delta x}$ can be decomposed as
\begin{equation*}
U_{\theta, \Delta x}=\widetilde{U}_{\theta, \Delta x}+\widebar{U}_{\theta, \Delta x},
\end{equation*}
where $\widetilde{U}_{\theta, \Delta x}$ is the approximate
solution obtained by solving homogeneous conservation laws in each time
step and $\widebar{U}_{\theta, \Delta x}$ is the perturbation
term; $\widetilde{U}_{\theta, \Delta x}$ and
its total variation are uniformly bounded. According to the results of
\cite{G,S2}, it can be accomplished that the
Glimm functional is nonincreasing in time.

Let $J$ be a mesh curve, $J'$ be the immediate successor of $J$, and $\Gamma_{k,n}$ be the diamond region enclosed by $J$ and $J'$, centered at $(x_{2k},t_n)$. The Glimm functional $F$ for $\widetilde{U}_{\theta ,\Delta x}$
over $J$ is defined as
\begin{equation}
\label{glimfunl}
F(J):=L(J)+KQ(J),
\end{equation}
where $K$ is a sufficiently large constant, which will be determined later, and
\begin{align*}
L(J)&:=\sum  \{ |\a_i| : \a_i \  \mbox{crosses} \ J \}+K_1\Big(|\beta_1|+\sum_{k\in B(J)}l_B^k
 \Big),\\
Q(J)&:= \sum \{ |\a_i||\a_{i'}|: \a_i,\a_{i'} \ \mbox{cross} \  J \ \mbox{and approach} \},\\
l_B^n&:=|\rho_B^{n+1}-\rho_B^n|+|m_B^{n+1}-m_B^n|.
\end{align*}
Here, both constants $K>1$ and $K_1>1$ will be determined later, $B(J):=\{n:P_{x_B,n}=(x_B,t_n+\frac{\Delta t}{2})\in J\}$, $l_b^n$ is evaluated at the mesh point $P_{x_B,n}$, and the presence of $|\beta_1|$ is because $\beta_1$ crosses $J$ and locates in some boundary triangle region (see Figure 5).\\

We first consider the case that $J$ and $J'$ differ in the diamond region away from the boundary.
According to Theorem 3.1., let $Q(\Gamma_{k,n}):=D(\a,\beta)$ be the wave interaction potential associated with $\a$ and $\beta$ and let
$$
C_1:=\max\limits_{U\in\Omega}\Big|\sum\limits_{j<i}L_{ij}(U)\Big|\ge C'_{k,n},\ \forall\,k,n.
$$
By the condition $(A_2)$, $u_0(x)>0,\ \forall\,x\in[x_B,\infty)$. According to \eqref{3.6}, we have the following inequalities
 \begin{align}
\label{3.32}
L(J')-L(J)&\leq C_1Q(\Gamma_{k,n})-\l_*^{-1}\z_{k,n}(|\a|+|\beta|)\D x
+\l_*^{-1}C''_{k,n}(\Delta x)^2+O(1)(\Delta x)^3,\\
\label{3.33}
Q(J')-Q(J)&\leq -Q(\Gamma_{k,n})+L(J)\big(C_1Q(\Gamma_{k,n})-\l_*^{-1}\z_{k,n}(|\a|+|\beta|)\Delta{x}\nonumber\\
&\quad+\l_*^{-1}C''_{k,n}(\Delta x)^2+O(1)(\Delta x)^3\big),
\end{align}
where $C''_{k,n},\z_{k,n}$ are in \eqref{3.6} of Theorem \ref{thm3.1}.
By \eqref{glimfunl}, \eqref{3.32}, and \eqref{3.33},
\begin{align}
\label{3.34}
F(J')-F(J)&\leq -\big(K-C_1-KC_1L(J)\big)Q(\Gamma_{k,n})-\l_*^{-1}\z_{k,n}(|\a|+|\beta|)\D x\nonumber\\
&\quad+\big(1+KL(J)\big)\l_*^{-1}C''_{k,n}(\Delta x)^2+O(1)(\Delta x)^3.
\end{align}
If $K$ satisfies $2C_1<K\le\frac{\e}{L(J)}$ for some $0<\e<\frac{1}{2}$, then
\begin{align}
\label{3.38.c}
F(J)=L(J)+KQ(J)\leq L(J) + K L^2(J)\leq(1+\e)L(J).
\end{align}
Coupling \eqref{3.38.c} with \eqref{3.34}, we obtain the estimate
\begin{align}
\label{3.36}
F(J')&<F(J)-\l_*^{-1}\z_{k,n}(|\a|+|\beta|)\D x+\l_*^{-1}(1+\e)C''_{k,n}(\Delta x)^2+O(1)(\Delta x)^3.
\end{align}
Now, let $J_n$, $n=1,2,\ldots$, denote the mesh curves that contain all mesh points
$(x_k+\theta_{n-1}\Delta x,t_{n-1})$ at time $t=t_{n-1}$; therefore, $J_{n}$ is located on the time strip $T_n:=(x_B,\infty)\times[t_{n-1},t_n)$. We select the positive number $K$ such that
\begin{equation}
\label{3.35}
2C_1<K\le\frac{\e}{L(J_1)+\mathcal{C}}=\frac{\e}{\TV\{U_0\}+\mathcal{C}},
\end{equation}
where $\mathcal{C}$ will be determined later. Then, $2C_1<K\le\frac{\e}{L(J_1)}$.
By \eqref{3.38.c} and adding up recursive relation \eqref{3.36} over all $k$,
and using $\int_{x_B}^{\infty}\frac{dx}{x^2}=\frac{1}{x_B}$ and $q\in W^{1,1}[x_B,\infty)$, we obtain
\begin{align*}
F(J_{2})&< F(J_1)-\l_*^{-1}C_2F(J_1)\Delta x+\l_*^{-1}C_3\Delta x+O(1)(\Delta x)^2,
\end{align*}
where
\begin{equation*}
C_2=\frac{3-\g}{x_B}\cdot\min_{t\ge 0}\{u_B(t)\},\quad\ C_3=\sqrt{3}\cdot\max_{U\in\Omega}\Big\{\frac{m}{x_B}+\frac{2\r v_B^2}{x_Bc_*}+\frac{(\g-1)\|q'\|_{L^1[x_B,\infty)}}{c_*^2}\Big\}.
\end{equation*}
Therefore, if $\D x$ sufficiently small, we have
\begin{align*}
L(J_2)\leq F(J_{2})&<\Big(1-\l_*^{-1}\frac{C_2}{1+\e}\Delta x\Big)F(J_1)+\l_*^{-1}(1+\e)C_3\Delta x+O(1)(\Delta x)^2 \nonumber\\
      &< F(J_1)+(1+\e)^2\mathcal{C}+O(1)(\D x)^2 \nonumber\\
      &\le (1+\e)L(J_1)+(1+\e)^2\mathcal{C}+O(1)(\D x)^2,
\end{align*}
where
\begin{equation}
\label{const}
\mathcal{C}:=\frac{C_3}{C_2}
=\frac{\sqrt{3}(1+\e)^2}{(3-\g)\min\limits_{t\ge 0}\{u_B(t)\}}\cdot\max\limits_{U\in\Omega}\Big\{m+\frac{\r GM_p}{x_Bc_*}+\frac{(\g-1)x_B\|q'\|_{L^1[x_B,\infty)}}{c_*^2}\Big\},
\end{equation}
and $c_*$ is given in \eqref{minc}.
Define $m_*:=\min\limits_{t\ge 0}\{m_B(t)\}$. By the condition ($A_2$),
\begin{align*}
m(x,t_2)\geq m_B(t_2)-L(J_2)\geq m_*-(1+\e)L(J_1)-(1+\e)^2\mathcal{C}>0.
\end{align*}
This implies that $u_{\theta,\D x}(x,t_2)>0,\ \forall\,x\in(x_B,\infty)$.
Moreover, $L(J_2)\le(1+\e)L(J_1)+(1+\e)^2\mathcal{C}$ and
$$
KL(J_2)\le K\big(L(J_1)+(1+\e)^2\mathcal{C}\big)+\e KL(J_1)+O(1)(\D x)^2\le\e(1+\e).
$$
Therefore, $-K+C_1+KC_1L(J_2)<-(1-\e-\e^2)C_1<0$. According to \eqref{3.6} and \eqref{glimfunl}
and the similarly argument in the previous step, we further obtain
\begin{align*}
L(J_3)&\leq F(J_{3})\leq\Big(1-\l_*^{-1}\frac{C_2}{1+\e+\e^2}\D x\Big)F(J_2)+\l_*^{-1}(1+\e+\e^2)C_3\D x+O(1)(\Delta x)^2\nonumber\\
&\le \Big(1-\l_*^{-1}\frac{C_2}{1+\e+\e^2}\D x\Big)^2F(J_1)\nonumber\\
&\quad+\l_*^{-1}(1+\e+\e^2)C_3\D x\Big(1+\Big(1-\l_*^{-1}\frac{C_2}{1+\e+\e^2}\D x\Big)\Big)+O(1)(\Delta x)^2\nonumber\\
&\le F(J_1)+(1+\e+\e^2)^2\mathcal{C}+O(1)(\D x)^2\nonumber \\
&\le (1+\e)L(J_1)+(1+\e+\e^2)^2\mathcal{C}+O(1)(\D x)^2,
\end{align*}
and
\begin{align*}
m_{\theta,\D x}(x,t_3)\geq m_B(t_3)-L(J_3)\geq m_*-(1+\e)L(J_1)-(1+\e+\e^2)^2\mathcal{C}>0,
\end{align*}
which implies $u_{\theta,\D x}(x,t_3)>0,\ \forall\,x\in[x_B,\infty)$. Based on the selection of the constants $K$ and $\e$ in \eqref{3.35}, it can be verified that
\begin{align*}
KL(J_3)&\le K\big((1+\e)L(J_1)+(1+\e)^2\mathcal{C}\big)+\e KL(J_1)+\e^2(2+2\e+\e^2)K\mathcal{C}+O(1)(\D x)^2, \\
       &\le\e+\e K\big(L(J_1)+(1+\e)^2\mathcal{C}\big)+O(1)(\D x)^2, \\
       &\le\e(1+\e).
\end{align*}
Continue this process and by using induction, if \eqref{3.35} holds true for all $J$ with $J_k\leq J<J_{k+1}$, $k=1,\ldots,n-1$, it yields
\begin{align*}
F(J_{n})&\le\Big(1-\l_*^{-1}\frac{C_2}{1+\e+\e^2}\D x\Big)F(J_{n-1})+\l_*^{-1}(1+\e+\e^2)C_3\D x+O(1)(\Delta x)^2\nonumber\\
&\leq \Big(1-\l_*^{-1}\frac{C_2}{1+\e+\e^2}\D x\Big)^{n-1}F(J_1)\nonumber\\
&\quad+\l_*^{-1}(1+\e+\e^2)C_3\D x\sum_{k=1}^{n-1}\Big(1-\l_*^{-1}\frac{C_2}{1+\e+\e^2}\D x\Big)^{k-1}+O(1)(\Delta x)^2,
\end{align*}
and therefore,
\begin{align}
\label{3.42}
F(J_{n})&\le F(J_1)+(1+\e+\e^2)^2\mathcal{C}+O(1)(\Delta x)^2\nonumber\\
&\leq (1+\e)\TV\{U_0\}+(1+\e+\e^2)^2\mathcal{C}+O(1)(\Delta x)^2, \\
m_{\theta,\D x}(x,t_n)&\geq m_B(t_n)-L(J_n)\geq m_*-(1+\e)L(J_1)-(1+\e+\e^2)^2\mathcal{C}>0.\nonumber
\end{align}
In particular, $u_{\theta,\D x}(x,t_n)>0,\ \forall\,x\in[x_B,\infty)$.
Therefore, the inequality \eqref{3.42} leads to
\begin{align}
\label{3.43}
\TV_{J}\{\widetilde{U}_{\theta,\Delta x}\}&\leq O(1)L(J)\leq O(1)F(J)\nonumber\\
&\leq (1+\e)\TV\{U_0(x)\}+(1+\e+\e^2)^2\mathcal{C}+O(1)(\Delta x)^2
\end{align}
for $J_k\leq J<J_{k+1}$, $k=1,\ldots,n-1$.

Next, we consider the case that $J'$ is an immediate successor of $J$ so that they only
differ on boundary $P_{x_B,{n}}$.
According to the conditions ($A_2$), \eqref{4.26b}, and \eqref{glimfunl}, we obtain
\begin{align*}
F(J')-F(J)
&=|\varepsilon|-|\a|-|\beta_1|-K_1(|\beta_1|+l_B^k)\nonumber\\
&\quad+K|\a|(|\varepsilon|-|\a|-|\beta_1|)-K(|\a_0\beta_1|+|\a_2\beta_1|+|\a_3\b_1|),\nonumber\\
&\leq O(1)C(|\a_0\beta_1|+|\a_2\beta_1|+|\a_3\beta_1|+|\beta_1|+l_B^k)-K_1(|\beta_1|+l_B^k)\nonumber\\
&\quad+O(1)CK|\a|(|\a_0\beta_1|+|\a_2\beta_1|+|\a_3\beta_1|+|\beta_1|+l^k_B)\nonumber\\
&\quad-K(|\a_0\beta_1|+|\a_2\beta_1|+|\a_3\beta_1|)+O(\Delta x)^2\nonumber\\
&\leq (-K_1+O(1)C+O(1)CK\cdot F(J))(|\beta_1|+l_B^k)\nonumber\\
&\quad+(-K+O(1)C+O(1)CK\cdot F(J))(|\a_0||\beta_1|+|\a_2||\beta_1|+|\a_3||\beta_1|)\nonumber\\
&\quad+O(1)(\Delta x)^2\leq O(1)(\Delta x)^2
\end{align*}
provided that constants $K_1$, $K\geq 2C_1$, and $ KL(J)\leq\e$. Now, let $J_n$ be the mesh curve located on the time strip $T_n:=(x_B,\infty)\times[t_{n-1},t_n)$ and include the half-ray $\{x=x_B,\;t\geq t_n+\frac{\Delta t}{2}\}$. Moreover, let
$\TV\{U_0(x)\}:=\TV\{\r_0(x)\}+\TV\{m_0(x)\}+\TV\{E_0(x)\}$. If $\Delta x$ and $\TV\{U_0(x)\}$ are sufficiently small,
then we have
\begin{align}
\label{4.41b}
F(J_{k+1})\le F(J_{k})-\l_*^{-1}\frac{C_2}{1+\e+\e^2}(\D x)F(J_k)+O(1)(\Delta x)^2,\quad k=1,...,n.
\end{align}
Based on \eqref{4.41b} and the analysis similar to that in the interior wave interaction, we show that the Glimm functional $F$ is nonincreasing in time. Therefore, $\widetilde{U}_{\theta,\Delta x}$ is defined for $t > 0$ and $\Delta x \rightarrow 0$.

Next, we verify that the total variation of the perturbation
is bounded in any time step.
Let us denote $S_k:=S(x_k,t,\widetilde{U}_{\th,\D x}(x_k,t))$, where $S$ is given in \eqref{solver}. Then,
\begin{equation*}
\begin{split}
\TV\{\widebar{U}_{\theta,\Delta x}\}
&=\sum_k|\widebar{U}_{\theta,\D x}(x_{k+1})
  -\widebar{U}_{\theta,\D x}(x_{k-1})|\le\sum_k|\big((S-I_3)\widetilde{U}_{\theta,\D x}\big)(x_{k+1})
  -\big((S-I_3)\widetilde{U}_{\theta,\Delta x}\big)(x_{k-1})| \\
&\le\sum_k|(S_z-I_3)(\widetilde{U}_{\th,\D x}(x_k+1)-\widetilde{U}_{\th,\D x}(x_{k-1})| \\
&\quad+\sum_k|(S_{k+1}-S_z)\widetilde{U}_{\th,\D x}(x_{k+1})+(S_z-S_{k-1})\widetilde{U}_{\th,\D x}(x_{k-1})|
\end{split}
\end{equation*}
According to \eqref{3.17-2} and the definition of $h(x)=-\frac{2}{x}$, we obtain
\begin{align}
\label{3.45}
\TV\{\widebar{U}_{\th,\D x}\}
&\le\|S_z-I_3\|\sum_k\osc\{\widetilde{U}_{\th,\D x}\}+2\|W\|\sum_k\frac{1}{x_k^2}\D x\D t+\|\Psi\|\sum_k\osc\{\widetilde{U}_{\th,\D x}\} \nonumber\\
&\le\|S_z-I_3\|\TV\{\widetilde{U}_{\th,\D x}\}+\frac{2\|W\|}{x_B}\D t+\|\Psi\|\TV\{\widetilde{U}_{\th,\D x}\}.
\end{align}
Since $\TV\{\widetilde{U}_{\th,\D x}\}$ is finite, the total variation of $\widebar{U}_{\th,\D x}$ is bounded.
Because of the boundedness of the total variation of approximate solutions,
 the constant $\mathcal{C}$ in \eqref{const} can be easily determined by the initial and boundary data, heat profile, and gravity. By \eqref{3.43}, \eqref{3.45}, and the results in \cite{DH,S2}, the following theorem
 is achieved.

\begin{thm}
Let $K,\ \e$ be as chosen in \eqref{3.35}, and let $U_{\theta,\Delta x}$ be an approximate solution of \eqref{IBVP}{\rm} based on the generalized Glimm scheme. Then, under the condition ${\rm(}A_1{\rm)}\sim{\rm(}A_3{\rm)}$,
for any given constant state $\widecheck{U}$, there exists a positive constant $d$, depending on the radius $r$ of $\Omega$, such that if
\begin{align*}
\sup_{x \in [x_B,\infty)}|U_0(x)-\widecheck{U}|\leq\frac{r}{2}, \quad\TV\{U_0(x)\}\leq d,
\end{align*}
and the condition
\begin{align*}
\sup_{t\in{\mathbb{R}^+}}|m_B(t)-\widecheck{m}|\leq\frac{r}{2}+(1+\e+\e^2)^2\mathcal{C}
\end{align*}
hold true for \eqref{IBVP} with the constant $\mathcal{C}$ in \eqref{const},
then $U_{\theta,\Delta x}(x,t)$ is well-defined for $t \geq 0$ and $\Delta x>0$ is sufficiently small.
Furthermore, $U_{\th,\D x}(x,t)$ has a uniform bound on the total variation and satisfies the following properties:
\begin{enumerate}
\item [{\rm(}i{\rm)}] $\displaystyle \|U_{\theta,\Delta{x}}-\widecheck{U}\|_{L^\infty}\leq r+(1+\e+\e^2)^2\mathcal{C}$.
\item [{\rm(}ii{\rm)}] $\TV\{U_{\theta,\Delta{x}}(\cdot,t)\}\leq \dfrac{r}{2}+(1+\e+\e^2)^2\mathcal{C}$.
\item [{\rm(}iii{\rm)}] $\displaystyle\int_{x_B}^{\infty}|U_{\theta,\Delta{x}}(x,t_2)-U_{\theta,\Delta{x}}(x,t_1)|dx\leq
O(1)(|t_2-t_1|+\Delta{t})$.
\item [{\rm(}iv{\rm)}] The velocity $u_{\th,\D x}(x,t)>0,\ \forall\,(x,t)\in\Pi$.
\item [{\rm(}v{\rm)}] The density $\r_{\th,\D x}(x,t)\ge\varrho,\ \forall\,(x,t)\in\Pi$, where $\varrho$ is the constant in {\rm(}$A_1${\rm)}.
\end{enumerate}
\end{thm}
\begin{proof}
We prove (ii) first. Choose a fixed $d$ such that $(1+\e)d\le\frac{r}{2}$. According to \eqref{3.43}, if
\begin{equation}
\label{3.44_a}
\sup|U_0(x)-\widecheck{U}|\leq \frac{r}{2},\quad\TV\{U_0(x)\}\leq d,
\end{equation}
then, for a sufficiently small $\Delta x>0$,
$$
\TV_{J_n}\{U_{\th,\D x}\}\le(1+\e)\TV\{U_0(x)\}+(1+\e+\e^2)^2\mathcal{C}\le\frac{r}{2}+(1+\e+\e^2)^2\mathcal{C},
$$
where the constant $\mathcal{C}$ is defined in \eqref{const}. Next, for (i),
\begin{align}
\label{3.44}
\sup_{J_n}|U_{\theta,\Delta x}-\widecheck{U}|\leq \sup |U_0(x)-\widecheck{U}|+\TV_{J_n}\{U_{\theta,\Delta x}\}.
\end{align}
By \eqref{3.44_a} and \eqref{3.44},
\begin{align*}
\sup_{\scriptstyle J_n}|\widetilde{U}_{\th,\Delta x}-\widecheck{U}|\leq r+(1+\e+\e^2)^2\mathcal{C}.
\end{align*}
By using the aforementioned constants $K$ and $d$, we obtain that, for a sufficiently small $\Delta x$, $U_{\th,\Delta x}(x,t)$ is defined on $\Pi$ when the condition ${\rm(}A_1{\rm)}\sim{\rm(}A_3{\rm)}$ hold true. In addition, $U_{\th,\Delta x}(x,t)$ and its total variation are uniformly bounded and independent of $\Delta x$.

For (iii), without loss of generality, let $t_2>t_1$, $t_0=\sup\{t\le t_1\mid t=n\D t\text{ for some }n\}$, and let $\ell=\lfloor\frac{t_2-t_0}{\D t}\rfloor+1$. According to \eqref{2.13.1},
\begin{align*}
|U_{\th,\D x}(x,t_2)-U_{\th,\D x}(x,t_1)|
&\le |U_{\th,\D x}(y,t_0)-U_{\th,\D x}(x,t_0)|+|(S(y,\D t,\widetilde{U}_{\th,\D x}(y,t_0))^{\ell}-I)\widetilde{U}_{\th,\D x}(y,t_0)| \\
&=|U_{\th,\D x}(y,t_0)-U_{\th,\D x}(x,t_0)|+O(1)(\D t)
\end{align*}
for some $y\in[x-\ell\D x,x+\ell\D x]$. Therefore, according to the Corollary 19.8 in \cite{S2},  the result (iii)
is obtained immediately.
\end{proof}

According to Theorem 3.4. and Oleinik's analysis in \cite{S2},
the following theorem for the compactness of the subsequence
of $\{U_{\theta,\Delta x}\}$ holds true.

\begin{thm}
Assume that the condition ${\rm(}A_1{\rm)}\sim{\rm(}A_3{\rm)}$ hold true. Let $\{U_{\theta ,\Delta x}\}$ be a family of approximate solutions \eqref{IBVP} obtained using the generalized Glimm scheme. Then, there exist a subsequence
$\{U_{\theta ,\Delta x_i}\}$ of $\{U_{\theta ,\Delta x}\}$ and measurable function $U$ such that
\begin{enumerate}
\item [{\rm(}i{\rm)}] $U_{\theta ,\Delta x_i}(x,t)\rightarrow U(x,t)$ in $L^1_{loc}$ as
$\Delta x_i\rightarrow 0$;
\item [{\rm(}ii{\rm)}] for any continuous function $f$, we have $f(x,t,U_{\theta,\Delta x_i})\rightarrow f(x,t,U)$
in $L^1_{loc}$ as $\Delta x_i \rightarrow 0$.
\end{enumerate}
\end{thm}

\subsection{The consistency and existence of the entropy solution}

Finally, the global existence of entropy solutions to \eqref{IBVP}
is presented by demonstrating the consistency of the scheme and entropy inequalities for the weak solutions.
To achieve the consistency,  the convergence of the residual to the approximate
solutions $\{U_{\theta ,\Delta x}\}$ of \eqref{IBVP} is given.

\begin{thm}
Consider the problem \eqref{IBVP} with the condition ${\rm(}A_1{\rm)}\sim{\rm(}A_3{\rm)}$. Assume that $\{U_{\th,\D x}\}$ is a sequence of approximate solutions for \eqref{IBVP}, which are constructed using the generalized Glimm scheme. There exist a null set $N\subset \Phi$ and subsequence $\{\Delta x_i\}$ such that the limit $$U(x,t):=\displaystyle\lim_{\Delta x_i \rightarrow 0\atop\theta \in \Phi \backslash N} U_{\theta,\Delta{x}_i}(x,t)$$ is an entropy solution of \eqref{IBVP}.
\end{thm}
\begin{proof}
We first calculate the residual of $\{U_{\theta ,\Delta x}\}$.
According to \eqref{aproxsol4}, $U_{\th,\D x}=\widetilde{U}_{\th,\D x}+\widebar{U}_{\th,\D x}$. For convenience,
we omit the symbol $\th$ in $U,\ \widetilde{U}$, and $\widebar{U}$ in the rest of this subsection.
Define
$$
D_{k,n}:=[x_{k-1},x_{k+1}]\times[t_n,t_{n+1}],\ n=0,1,2,\cdots;k=1,2,\cdots,
$$
and
\begin{equation*}
\lfloor U\rfloor(x,t_n):=U(x,t_n^+)-U(x,t_n^-).
\end{equation*}
We concentrate on the case that $D_{k,n}$ is away from the boundary, that is, $k\ge 1$; the case $k=0$ can be estimated similarly. For a test function
$\phi(x,t)\in C^1_0(\Pi)$,
by summing \eqref{RPres} over all $D_{k,n}$, we have the following estimate
$\{U_{\Delta x}\}$:
\begin{align}
\label{4.1}
&\iint_{x>x_B,t>0}\{U_{\Delta{x}}\phi_t+f(U_{\Delta{x}})\phi_x+h(x)g(x,U_{\Delta{x}})\phi\}dx dt
 +\int^{\infty}_{x_B} U_0(x)\phi(x,0) dx+\int^{\infty}_{0} f(\widehat{U}_B(t))\phi(x_B,t) dt\nonumber\\
&=\sum_{\scriptstyle n=0}^\infty\sum_{k=1}^{\infty}R(U_{\Delta{x}}, D_{k,n}, \phi)+\int_{x_B}^\infty U_0(x)\phi(x,0) dx
 +\int_{0}^\infty f(\widehat{U}_B(t))\phi(x_B,t) dt\nonumber\\
&=-\sum_{n=1}^\infty\sum_{k=1}^{\infty}\int^{x_{k+1}}_{x_{k-1}}\lfloor\widetilde{U}_{\Delta{x}}+{\widebar{U}_{\Delta{x}}}\rfloor(x, t_n)
  \phi(x, t_n)dx\nonumber\\
&\quad-\int^{\infty}_{x_B}\big(\widetilde{U}_{\Delta{x}}(x,0^+)-U_0(x)\big)\phi(x,0)dx
 -\int^{\infty}_{0}\big(f(\widetilde{U}_{\Delta{x}}(x_B^+,t))-f(\widehat{U}_B(t))\big)\phi(x_B,t)dt\nonumber\\
&\quad+O(1)(\Delta{x})(1+T.V.\{U_0,\widehat{U}_B\})\|\phi\|_\infty\nonumber\\
&=J(\theta , \Delta x , \phi)+O(1)\Delta x ,
 \end{align}
where
$$
J(\theta , \Delta x ,
\phi):=-\sum_{n=1}^\infty\sum_{k=1}^{\infty}\int^{x_{k+1}}_{x_{k-1}}\lfloor\widetilde{U}_{\Delta{x}}+\widebar{U}_{\Delta{x}}\rfloor(x, t_n)\phi(x, t_n)dx.
$$
According to Glimm's argument in \cite{G} and
\eqref{2.13.1},
\begin{align}
\label{4.2} J(\theta , \Delta x , \phi) \rightarrow 0\; \mbox{
as }\Delta{x}\rightarrow 0
\end{align}
for almost random sequence $\theta \in \Theta$, where $\Theta$ is a
probability space of the random sequence. Therefore, by \eqref{4.1}, \eqref{4.2} and Theorem 3.3.,
there exist a null set $N\subset \Phi$ and subsequence $\{\Delta x_i\}$ such
that the limit $$U(x,t):=\displaystyle\lim_{\Delta x_i \rightarrow
0\atop\theta \in \Phi \backslash N} U_{\theta,\Delta{x}_i}(x,t)$$ is
a weak solution
of \eqref{IBVP}.

Next, we show that the aforementioned weak solution $U$ is indeed an entropy solution satisfying the entropy inequality \eqref{ibvpentropy}.
Given an entropy pair $(\eta,\om)$, define
\begin{align*}
\widehat{R}(U_{\Delta{x}},\Pi, \phi)
:=\displaystyle\iint_{x>x_B,t>0}\left\{\eta({U}_{\Delta{x}})\phi_t+\om({U}_{\Delta{x}})\phi_x+d\eta({U}_{\D x})\cdot h(x)g(x,U_{\D x})\phi\right\}
dxdt.
\end{align*}
By using an argument similar to the proof of Theorem 2.2. (see Appendix B), we can estimate
the residuals of $U_{\D x}$ in $D_{k,n}$ as
\begin{align}
\label{RPres1}
&\widehat{R}(U_{\D x},D_{k,n},\phi)\ge\int^{x_{k+1}}_{x_{k-1}}\big(d\eta(\widetilde{U}_{\D x})\widebar{U}_{\D x}\p\big)(x,t^-_{n+1})dx
-\int^{x_{k+1}}_{x_{k-1}}\big(\eta(\widetilde{U}_{\D x})\p\big)(x,t)\Big|^{t=t^-_{n+1}}_{t=t^+_n}dx\nonumber\\
&\quad+\int^{t_{n+1}}_{t_n}\big(\om(U_{\D x})\p\big)(x,t)\Big|^{x=x_{k+1}}_{x=x_{k-1}}dt
+O(1)\left((\Delta t)^2(\Delta x)+(\Delta{t})^3+(\Delta{t})^2\underset{D_{k,n}}{\osc}\{\widetilde{U}\}\right)\|\phi\|_\infty.
\end{align}
Summing \eqref{RPres1} over all $D_{k,n}$, we have
\begin{align}
\label{4.15}
&\sum_{\scriptstyle n=0}^\infty\sum_{k=1}^{\infty}\widehat{R}(U_{\Delta{x}}, D_{k,n}, \phi)+\int_{-\infty}^\infty \eta(U_0(x))\phi(x,0)dx\nonumber\\
&\geq-\sum_{n=1}^\infty\sum_{k=1}^{\infty}\int^{x_{k+1}}_{x_{k-1}}\lfloor\eta(\widetilde{U}_{\Delta{x}})
 +d\eta(\widetilde{U}_{\Delta x})\widebar{U}_{\Delta x}\rfloor(x,t_n)\phi(x, t_n)dx\nonumber\\
&\quad-\int^{\infty}_{-\infty}\big(\eta(\widetilde{U}_{\Delta{x}}(x, 0^+))-\eta\left(U_0(x)\right)\big)\phi(x, 0)dx+O(\Delta{x})
\end{align}
for a sufficiently small $\Delta x$. Based on Glimm's argument, \eqref{4.15} implies that
\begin{align*}
&\iint_{x>x_B,t>0}\left\{\eta(U)\phi_t+\om(U)\phi_x+{d\eta}\cdot h(x)g(x,U)\phi\right\} dxdt
+\int^\infty_{x_B}\eta(U_0(x))\phi(x,0)dx\nonumber\\
&\hspace{9cm}+\int_0^{\infty}\om(\widehat{U}_B(t))\p(x_B,t)dt\geq 0
\end{align*}
for every entropy pair $(\eta,\om)$ and positive test function $\phi\in C^1_0(\Pi)$. The existence of the entropy solution of \eqref{IBVP} is established.
\end{proof}

\section{Hydrodynamic regions}
\setcounter{equation}{0}

In Section 3, the global existence of the entropy solution to HEP \eqref{HEP1} is established on the basis of the initial density with a positive lower bound. However, \eqref{HEP1} does not fulfill the physical meaning because of which the atmosphere density reaches vacuum as $x$ approaches infinity. Therefore, it is necessary to determine the hydrodynamic region, a subset of $\Pi\equiv[x_B,\infty)\times[0,\infty)$, in which our solution of \eqref{HEP1} is physically well-defined. This section is devoted to establishing the main Theorem II. For a certain constraint on the transonic initial data $U_0(x)=(\r_0(x),m_0(x),E_0(x))^T$, where $\r_0,E_0$ is a decreasing function and $m_0$ is an increasing function in $[x_B,\infty)$, there exists a region $\Sg_1\equiv[x_B,x^*]\times[0,\infty)$ such that the wave speeds of the solutions to \eqref{IBVP} are positive in $\Pi\backslash\Sg_1$. Next, by adopting the Knudsen number of the gas, we prove that there also exists a region $\Sg_2=[x_B,x^{**}]\times[0,\infty)$ such that the gas in $\Pi\backslash \Sg_2$ no longer acts like fluid. In other words, \eqref{HEP1} fails to model the HEP outside $\Sg_2$, and the governed equations must be replaced by the kinetic equations. According to the assumed initial data described previously, there exists a nonempty hydrodynamic region of HEP \eqref{HEP1} such that $\Sg\equiv\Sg_1\cup\Sg_2$.

\begin{lemma}
Let $U(x,t)=(\r(x,t),m(x,t),E(x,t))^T$ be the solution of \eqref{IBVP} constructed in Theorem 3.5., with initial data $U_0(x)=(\r_0(x),m_0(x),E_0(x))^T$ and boundary data $(\r_B(t),m_B(t))^T$. Then, for any $(x,t)\in\Pi$, we have
\begin{equation}
\label{globlestimate}
\begin{split}
&|\r(x,t)-\r_0(x)|\le\TV\{\r_0(x)\}+\TV\{\r_B(t)\}, \\
&|m(x,t)-m_0(x)|\le\TV\{m_0(x)\}+\TV\{m_B(t)\}, \\
&|E(x,t)-E_0(x)|\le\TV\{E_0(x)\}+\TV\{E(x_B,t)\}+2\l_*^{-1}\|q\|_{L^1[x_B,\infty)}.
\end{split}
\end{equation}
\end{lemma}
\begin{proof}
For any $(x,t)\in\Pi$, let $k=\lfloor\frac{x}{\D x}\rfloor,\ n=\lfloor\frac{t}{\D t}\rfloor+1$, and let $D_{k,n}$ denote the Riemann cell containing the point $(x,t)$. By \eqref{aproxsol4} and the random choice process, the approximate solution in the $n$th time step satisfies
\begin{equation*}
\begin{split}
&|U_{\th,\D x}(x,n\D t)-U_{\th,\D x}(x,(n-1)\D t)| \\
&\le\big|\big(S(y,(n-1)\D t,\widetilde{U}_{\th,\D x}(y,(n-1)\D t))-I\big)\widetilde{U}_{\th,\D x}(y,(n-1)\D t)\big| \\
&\quad+|\widetilde{U}_{\th,\D x}(y,(n-1)\D t)-U_{\th,\D x}(x,(n-1)\D t)|,
\end{split}
\end{equation*}
for some $y\in D_{k,n}$ such that $U_{\th,\D x}(x,n\D t)=S(y,(n-1)\D t,\widetilde{U}_{\th,\D x}(y,(n-1)\D t))\widetilde{U}_{\th,\D x}(y,(n-1)\D t)$. More precisely, according to \eqref{aproxsol4},
\begin{equation*}
\begin{split}
\r_{\th,\D x}(x,n\D t)&=\tilde{\rho}_{\th,\D x}(y,(n-1)\D t)-\frac{2\tilde{m}_{\th,\D x}(y,(n-1)\D t)}{y} \Delta t +O(1)(\Delta t)^2,\\
m_{\th,\D x}(x,n\D t) &=\tilde{m}_{\th,\D x}(y,(n-1)\D t) \\
&\quad-\frac{2\tilde{\rho}_{\th,\D x}(y,(n-1)\D t)}{y}\big(\tilde{u}_{\th,\D x}(y,(n-1)\D t)^2+v(y)^2\big)\Delta t +O(1)(\Delta t)^2,\\
E_{\th,\D x}(x,n\D t)&=\widetilde{E}_{\th,\D x}(y,(n-1)\D t)+q(y)\D t \\
&\quad-\frac{2\tilde{m}_{\th,\D x}(y,(n-1)\D t)}{y}\Big(\frac{1}{2}\tilde{u}_{\th,\D x}(y,(n-1)\D t)^2+v(y)^2+\frac{\tilde{c}_{\th,\D x}(y,(n-1)\D t)^2}{\g-1}\Big)\Delta t \\
&\quad+O(1)(\Delta t)^2,
\end{split}
\end{equation*}
where $v(y)=\sqrt{\frac{GM_p}{2y}}$. Since the density and momentum are globally positive according to Theorem 3.3.,
\begin{eqnarray*}
&& |\r_{\th,\D x}(x,n\D t)-\r_{\th,\D x}(x,(n-1)\D t)|\le\underset{D_{k,n}}{\osc}\{\tilde{\r}_{\th,\D x}\}, \\
&& |m_{\th,\D x}(x,n\D t)-m_{\th,\D x}(x,(n-1)\D t)|\le\underset{D_{k,n}}{\osc}\{\tilde{m}_{\th,\D x}\}, \\
&& |E_{\th,\D x}(x,n\D t)-E_{\th,\D x}(x,(n-1)\D t)|\le\underset{D_{k,n}}{\osc}\{\widetilde{E}_{\th,\D x}\}+[q(x)+q(y)]\D t,
\end{eqnarray*}
where $q$ is the heat profile. Note that the grid $D_{k,n}$ is the domain of dependence of $(x,t)$ for one time step. By using backward induction on $n$ and passing to the limit as $\D x\to 0,\ \th\in\Phi \backslash N$ in Theorem 3.5., the result \eqref{globlestimate} is obtained.
\end{proof}

Next, we show that there exist $x^*>x_B$ and $\Sg_1\equiv[x_B,x^*]\times[0,\infty)$ such that the wave speeds of $U(x,t)|_{\Pi\backslash \Sg_1}$ are positive. Recall that the Mach number of $U$ is defined as
$$
\mathfrak{Ma}(U):=\frac{|u|}{c}
$$

\begin{thm}
Assume that the transonic initial data $U_0=(\r_0,m_0,E_0)^T$ satisfying the condition ${\rm(}A_1{\rm)}\sim{\rm(}A_2{\rm)}$, where $\r_0,E_0$ are decreasing, $m_0$ is increasing, and $u_0(x_B)<c_0(x_B)$. Let $U=(\r,m,E)^T$ be the solution of \eqref{IBVP} constructed using Theorem 3.5.; then there exist $x^*\in(x_B,\infty)$ and $\Sg_1\equiv[x_B,x^*]\times[0,\infty)$ such that the characteristic speeds of the solution $U(x,t)$ in $\Pi\backslash\Sg_1$ are positive.
\end{thm}
\begin{proof}
According to \eqref{2.5.5}, the Mach number $\mathfrak{Ma}$ satisfies
$$
\mathfrak{Ma}(x,t)^2=\frac{u(x,t)^2}{c(x,t)^2}=\frac{m(x,t)^2}{\g(\g-1)(\r(x,t)E(x,t)-m(x,t)^2)}\ge\frac{m(x,t)^2}{\g(\g-1)\r(x,t)E(x,t)}.
$$
Let $E^*:=\max\limits_{x\in[x_B,\infty)}E_0(x)+\TV\{E_0\}+\TV\{E(x_B,t)\}+2\l_*^{-1}\|q\|_{L^1[x_B,\infty)}$, where $\l_*$ is defined in \eqref{CFL}.
According to Lemma 4.1., $E(x,t)\le E^*$ for all $(x,t)\in\Pi$ and therefore,
$$
\mathfrak{Ma}(x,t)\ge\frac{m(x,t)}{\sqrt{\g(\g-1)E^*\r(x,t)}}\text{ for all }(x,t)\in\Pi.
$$
Next, we define
$$
\mathfrak{M}(x):=\frac{m_0(x)-\TV\{m_0\}-\TV\{m_B\}}{\sqrt{\g(\g-1)E^*(\r_0(x)+\TV\{\r_0\}+\TV\{\r_B\})}}.
$$
According to Lemma 4.1.,
\begin{equation*}
\mathfrak{Ma}(x,t)\ge\mathfrak{M}(x),\ \forall\,(x,t)\in\Pi.
\end{equation*}
The function $\mathfrak{M}(x)$ is increasing based on the assumption of $\r_0$ and $m_0$. Since the initial data is transonic with $u_0(x_B)<c_0(x_B)$, $\mathfrak{M}(x_B)\le\mathfrak{Ma}(x_B,0)<1$. On the other hand, according to ($A_1$), the function $\mathfrak{M}(x)$ is greater than 1 in the far field, when $\r_0$ is near $\vr$ which is sufficiently small. There exists $x^*\in(x_B,\infty)$ such that $\mathfrak{M}(x^*)=1$, and $\mathfrak{M}(x)>1$ for $x\in(x^*,\infty)$. Therefore,
\begin{eqnarray*}
&&\mathfrak{Ma}(x^*,t)\ge\mathfrak{M}(x^*)=1,\quad \text{for }t\in[0,\infty), \\
&&\mathfrak{Ma}(x,t)\ge\mathfrak{M}(x)>1,\quad \text{for }(x,t)\in(x^*,\infty)\times[0,\infty).
\end{eqnarray*}
Denote the region $\Sg_1\equiv[x_B,x^*]\times[0,\infty)$. We have shown that $u(x,t)\ge c(x,t)$, for all $(x,t)\in\Pi\backslash\Sg_1$, that is, the characteristic speeds of the solution $U(x,t)$ in $\Pi\backslash\Sg_1$ are positive.
\end{proof}

Next, we want to determine the hydrodynamic region $[x_B,x_T]$ such that the constructed solution $U(x,t)$ has mathematical and physical significance. The Knudsen number is defined as $\mathfrak{Kn}=\frac{l}{\mathcal{H}}$, the ratio of the mean free path of the molecules, $l=\frac{1}{\sqrt{2}\t\mathfrak{n}}$, to the density scale height, $\mathcal{H}=\frac{k_BT}{GM_p\mathfrak{m}/x^2}$, of the atmosphere. The region of validity of the hydrodynamic equation is often classified using the Knudsen number $\mathfrak{Kn}$, which is useful for determining whether statistical mechanics or continuum mechanics formulation of fluid dynamics must be used. Here, $\t\approx 10^{14}\pi\ \text{cm}^2$ is the collision cross section used in \cite{LIANG}, $\mathfrak{n}$ is the number density, $\mathfrak{m}$ is the mass of a molecule, $T$ is the temperature of the gas, and $k_B$ is the Boltzmann constant. According to \eqref{ideal},
$$
P=\r RT=\mathfrak{n}k_BT.
$$
The Knudsen number can be computed as follows:
\begin{equation*}
\mathfrak{Kn}(x,U)=\frac{l}{\mathcal{H}}=\frac{GM_p\mathfrak{m}}{\sqrt{2}\t\mathfrak{n}x^2k_BT}=\frac{GM_p\mathfrak{m}}{\sqrt{2}\t x^2P}
=\frac{\g GM_p\mathfrak{m}}{\sqrt{2}\t x^2\r c^2}.
\end{equation*}

The hydrodynamic equations are applied appropriately, where $\mathfrak{Kn}<1$, so that many collisions occur over relevant length scales keeping the gas in thermal equilibrium. If $\mathfrak{Kn}\ge 1$, the continuum assumption of fluid mechanics maybe no longer be a good approximation because there are few collisions in this level to inhibit a molecule from escaping.

\begin{thm}
For a transonic initial data $U_0=(\r_0,m_0,E_0)^T$ as in Theorem 4.2. and $\mathfrak{Kn}(x_B,0)<1$,
\begin{equation}
\label{knudsencond}
x_B^2\r_0(x_B)c_0(x_B)^2>\frac{\g GM_p\mathfrak{m}}{\sqrt{2}\,\t}.
\end{equation}
Let $U=(\r,m,E)^T$ be the solution of \eqref{IBVP} constructed using Theorem 3.5.; then, there exist $x^{**}\in(x_B,\infty)$ and $\Sg_2\equiv[x_B,x^{**}]\times[0,\infty)$ such that $\mathfrak{Kn}(x,t)\le 1,\ \forall\,(x,t)\in\Sg_2$.
\end{thm}
\begin{proof}
Let us denote $\psi(x,t):=\r(x,t)c(x,t)^2$ and $\psi_0(x):=\r_0(x)c_0(x)^2$; then, $\psi$ and $\psi_0$ are of bounded variation functions based on the construction of solutions in Theorem 3.5., and the lower boundedness of the density in ${\rm(}A_1{\rm)}$. Moreover, according to \eqref{2.5.5} and the assumption of the initial data,
\begin{equation*}
\psi'_0(x)=\frac{\g(\g-1)m_0^2}{2\r_0^2}\r_0'(x)-\frac{\g(\g-1)m_0}{\r_0}m_0'(x)+\g(\g-1)E_0'(x)<0,
\end{equation*}
that is, $\psi_0(x)$ is a decreasing function of $x$. Define the function $\psi_*(x)$ as
\begin{equation}
\label{defpsi}
\psi_*(x):=\psi_0(x)-\TV\{\psi_0(x)\}-\TV\{\psi(x_B,t)\}-2\l_*^{-1}\g(\g-1)\|q\|_{L^1[x_B,\infty)}.
\end{equation}
According to \eqref{aproxsol4}, for any $(x,t)\in\Pi$, the approximation of $\psi(x,t)$ in the $n$th time step can be evaluated as
\begin{equation*}
\begin{split}
\psi_{\th,\D x}(x,n\D t)&=\widetilde{\psi}_{\th,\D x}(y,(n-1)\D t)+\g(\g-1)q(y)\D t \\
&\quad-\frac{2\g\tilde{m}_{\th,\D x}(y,(n-1)\D t)\tilde{c}_{\th,\D x}(y,(n-1)\D t)^2}{y}\D t+O(1)(\Delta t)^2,
\end{split}
\end{equation*}
where $\psi_{\th,\D x}(x,n\D t):=\r_{\th,\D x}(x,n\D t)c_{\th,\D x}(x,n\D t)^2$ and $k,n$, and $y$ are as defined in Lemma 4.1. Based on the similar argument in Lemma 4.1 and \eqref{defpsi}, $\psi_*$ is a decreasing function and $\psi_*(x)\le\psi(x,t),\ \forall\,(x,t)\in\Pi$.
Next, define
$$
\mathfrak{K}(x):=\frac{\g GM_p\mathfrak{m}}{\sqrt{2}\t x_B^2\psi_*(x)}.
$$
Since $\psi_*(x)\le\psi(x,t)$ for all $(x,t)\in\Pi$,
$$
\mathfrak{Kn}(x,t)\le\frac{\g GM_p\mathfrak{m}}{\sqrt{2}\t x_B^2\psi(x,t)}\le\mathfrak{K}(x),\ \forall\,(x,t)\in\Pi.
$$
According to \eqref{defpsi} and the assumption \eqref{knudsencond}, $\mathfrak{K}(x_B)<1$. On the other hand, according to the decrease of the function $\psi_*$ and the assumption ($A_1$), $\mathfrak{K}(x)>1$ in the far field whenever $\r_0$ near $\vr$. There exists $x^{**}\in(x_B,\infty)$ such that $\mathfrak{K}(x^{**})=1$, and $\mathfrak{K}(x)\le1$ for $x\in[x_B,x^{**}]$. Therefore,
\begin{eqnarray*}
&&\mathfrak{Kn}(x^{**},t)\le\mathfrak{K}(x^{**})=1,\quad \text{for }t\in\times[0,\infty), \\
&&\mathfrak{Kn}(x,t)\le\mathfrak{K}(x)<1,\quad \text{for }(x,t)\in[x_B,x^{**})\times[0,\infty).
\end{eqnarray*}
Denote the region $\Sg_2\equiv[x_B,x^{**}]\times[0,\infty)$. Therefore, the Knudsen number $\mathfrak{Kn}(x,t)\le 1$ for all $(x,t)\in\Sg_2$.
\end{proof}

Finally, according to Theorem 4.2. and Theorem 4.3, we define the hydrodynamic region of HEP \eqref{HEP1} by using $\Sg\equiv\Sg_1\cup\Sg_2$. The wave speeds of the solution $U(x,t)$, constructed in Theorem 3.5., are positive in the region $\Pi\backslash\Sg$. Moreover, the Knudsen number $\mathfrak{Kn}(x,U)\le 1$ in the region $\Sg$. Therefore, we obtain an entropy solution $U(x,t)$ of \eqref{HEP1} that has both mathematical and physical significance in the hydrodynamic region $\Sg$.

\newpage
\appendix
\appendixpage

\section{Construction of the Riemann Solver}
\setcounter{equation}{0}

Here, we describe the construction of the Riemann solver in \eqref{aproxsol4}. Based on Theorem 2.1., let $\widetilde{U}$ be the entropy solution for $\mathcal{R}_C(x_0,t_0)$,
we only need to construct the perturbation $\widebar{U}$ of $U$ from $\widetilde{U}$. Let us denote
$$
\tilde{u}:=\frac{\tilde{m}}{\tilde{\r}},\quad v(x):=\sqrt{\frac{GM_p}{2x}},
$$
and consider the linearized system of \eqref{3x3system} around $\widetilde{U}$ with initial data $\widebar{U}(x,0)= 0$:
\begin{equation}
\label{perturb}
\left\{\begin{split}
&\widebar{U}_t+(A(x,t)\widebar{U})_x=B(x,t)\widebar{U}+C(x,t), \\
&\widebar{U}(x,0)=0,
\end{split}\right.\quad (x,t)\in D(x_0,t_0).
\end{equation}
where
\begin{eqnarray*}
&& A(x,t)=df(\widetilde{U})=\left[\begin{array}{ccc}
0 & 1 & 0 \\
\frac{\g-3}{2}\tilde{u}^2 & (3-\g)\tilde{u} & \g-1 \\
\vspace{-0.4cm} \\
\frac{\g-1}{2}\tilde{u}^3-\tilde{u}\widetilde{H} & \widetilde{H}-(\g-1)\tilde{u}^2 & \g\tilde{u}
\end{array}\right], \\
&& B(x,t)=h(x)g_U(x,\widetilde{U})=h(x)\left[\begin{array}{ccc}
0 & 1 & 0 \\
v^2-\tilde{u}^2 & 2\tilde{u} & 0 \\
\frac{\g-1}{2}\tilde{u}^3-\tilde{u}\widetilde{H} & \widetilde{H}-(\g-1)\tilde{u}^2-v^2 & \g\tilde{u}
\end{array}\right], \\
&& C(x,t)=h(x)g(x,\widetilde{U})=h(x)\left[\begin{array}{c}
\tilde{m} \\ \tilde{\r}(\tilde{u}^2+v^2) \\ \tilde{m}(\widetilde{H}+v^2)-\frac{xq}{2}
\end{array}\right],
\end{eqnarray*}
and $\widetilde{H}=H(\widetilde{U})$ is given in \eqref{enthalpy}. Due to the appearance of shocks or discontinuity in $\widetilde{U}$, the coefficients in \eqref{perturb} may be discontinuous. To obtain the better regularity of approximate solutions for \eqref{perturb}, the averaging process of the coefficients in \eqref{perturb} with respect to $t$ over $[0,\D t]$ is used. For a bounded variation function $z(x,t,U)$, the average of $z(x,t,U)$ is defined as
\begin{equation*}
z_{\star}(x):=\frac{1}{\Delta t}\int_{t_0}^{t_0+\Delta t}z(x,s,U(x,s))ds, \quad x_0-\Delta x\leq x\leq x_0+\Delta x.
\end{equation*}
In addition, $w_{\star}(x)$ is continuous even across the shock and contact discontinuity.

We construct approximate solution for \eqref{perturb} by solving
\begin{equation}
\label{avgpde}
\left\{\begin{split}
&(\widebar{U}_\star)_t+(A_\star(x)\widebar{U}_\star)_x=B_\star(x)\widebar{U}_\star+C_\star(x), \\
&\widebar{U}_\star(x,0)=0,
\end{split}\right.\quad (x,t)\in D(x_0,t_0).
\end{equation}
Based on the operator-splitting method, solutions for \eqref{avgpde} can be approximated through composing solutions for
\begin{equation}
\label{avghomo}
\left\{\begin{split}
&V_t+(A_\star(x)V)_x=0, \\
&V_\star(x,0)=0,
\end{split}\right.\quad (x,t)\in D(x_0,t_0),
\end{equation}
and
\begin{equation}
\label{avgode}
\left\{\begin{split}
&Y_t=B_{\star}(x)Y+C_{\star}(x), \\
&Y(x,0)=V(x,\D t),
\end{split}\right.\quad (x,t)\in D(x_0,0).
\end{equation}
Since the system \eqref{avghomo} admits zero solution. The solution for \eqref{perturb} can be approximated as the solution of the follows:
\begin{equation}
\label{splitode}
\left\{\begin{split}
&(\widebar{U}_\star)_t=B_\star(x)\widebar{U}_\star+C_\star(x), \\
&\widebar{U}_\star(x,0)=0,
\end{split}\right.\quad (x,t)\in D(x_0,t_0),
\end{equation}
where $B_{\star}(x)$ and $C_{\star}(x)$ can be derived from $\widetilde{U}_{\star}$. The averaging process is reasonable because
\begin{equation*}
\iint_{D(x_0,t_0)}|\widebar{U}-\widebar{U}_{\star}|dxdt=O(1)\left((\Delta
t)^3+(\Delta t)^2\cdot\TV_{D(x_0,t_0)}\{\widetilde{U}\}\right).
\end{equation*}

The matrix $B_{\star}(x)$ in \eqref{splitode} has eigenvalues
\begin{equation*}
\s_1(x)=h(\tilde{u}_{\star}+v),\quad\s_2(x)=\g h\tilde{u}_{\star},\quad\s_3(x)=h(\tilde{u}_{\star}-v).
\end{equation*}
and the corresponding eigenvectors
$$
P_1(x)=(1,\tilde{u}_{\star}+v,\D_1)^T,\quad P_2(x)=(0,0,1)^T,\quad P_3(x)=(1,\tilde{u}_{\star}-v,\D_2)^T.
$$
where
\begin{eqnarray*}
\D_1 &=& \frac{v^3}{v-(\g-1)\tilde{u}_{\star}}+\frac{2v^2-(\g-1)\tilde{u}_{\star}(\tilde{u}_{\star}+3v)}{2(v-(\g-1)\tilde{u}_{\star})}\tilde{u}_{\star}
      +\frac{\g v}{\tilde{\r}_{\star}(v-(\g-1)\tilde{u}_{\star})}\widetilde{E}_{\star}, \\
\D_2 &=& \frac{v^3}{v+(\g-1)\tilde{u}_{\star}}-\frac{2v^2-(\g-1)\tilde{u}_{\star}(\tilde{u}_{\star}-3v)}{2(v+(\g-1)\tilde{u}_{\star})}\tilde{u}_{\star}
      +\frac{\g v}{\tilde{\r}_{\star}(v+(\g-1)\tilde{u}_{\star})}\widetilde{E}_{\star}.
\end{eqnarray*}
The transformation matrix of $B_*(x)$ is then given by $P(x)=[P_1(x),P_2(x),P_3(x)]$ so that $B_*(x)$ can be diagonalized as
\begin{align*}
\L(x)=P^{-1}(x)B_{\star}(x)P(x)=h(x)\diag[\tilde{u}_{\star}+v,\g\tilde{u}_{\star},\tilde{u}_{\star}-v].
\end{align*}
Moreover, the fundamental matrix solved using $\dot{X}(t)=\L(x)X(t)$ is
\begin{align*}
e^{\L(x)t}=\diag[e^{h(\tilde{u}_{\star}+v)t},e^{\g h\tilde{u}_{\star}t},e^{h(\tilde{u}_{\star}-v)t}].
\end{align*}
Using the transformation matrix $P$,  the state transition matrix of $\dot{X}(t)=B(x)X(t)$ is given by
\begin{equation}
\label{transmatrx}
N(x,t,s)=P(x)e^{\L(x)(t-s)}P(x)^{-1}=N_1(x)e^{h(\tilde{u}_{\star}+v)(t-s)}+N_2(x)e^{\g h\tilde{u}_{\star}(t-s)}+N_3(x)e^{h(\tilde{u}_{\star}-v)(t-s)},
\end{equation}
where
\begin{equation*}
\begin{split}
N_1(x)
&=\frac{1}{2v}\left[\begin{array}{ccc}
-(\tilde{u}_{\star}-v) & 1 & 0 \\
-(\tilde{u}_{\star}^2-v^2) & v+\tilde{u}_{\star} & 0 \\
-(\tilde{u}_{\star}-v)\D_1 & \D_1 & 0
\end{array}\right], \\
N_2(x)
&=\frac{1}{2v}\left[\begin{array}{ccc}
0 & 0 & 0 \\
0 & 0 & 0 \\
(\tilde{u}_{\star}-v)\D_1-(\tilde{u}_{\star}+v)\D_2 & \D_2-\D_1 & 2v
\end{array}\right], \\
N_3(x)
&=\frac{1}{2v}\left[\begin{array}{ccc}
\tilde{u}_{\star}+v & -1 & 0 \\
\tilde{u}_{\star}^2-v^2 & -(\tilde{u}_{\star}+v) & 0 \\
(\tilde{u}_{\star}+v)\D_2 & -\D_2 & 0
\end{array}\right].
\end{split}
\end{equation*}
According to \eqref{transmatrx} and the variation of constant formula,
the solution for \eqref{splitode} is given by
\begin{eqnarray}
\label{avgperturb}
&& \hspace{-0.7cm}\widebar{U}_{\star}(x,t)=\int_{t_0}^{t_0+t}N(x,t,s)C_{\star}(x)ds
=\sum_{i=1}^3N_i(x)C_{\star}(x)\int_{t_0}^{t_0+t}e^{\s_i(x)(t-s)}ds \nonumber \\
&& \hspace{0.55cm}=\frac{\tilde{\r}_{\star}}{2}(e^{h(\tilde{u}_{\star}+v)t}-1)\left[\begin{array}{c}
1 \\ \tilde{u}_{\star}+v \\ \D_1
\end{array}\right]+\frac{\tilde{\r}_{\star}}{2}(e^{h(\tilde{u}_{\star}-v)}-1)\left[\begin{array}{c}
1 \\ \tilde{u}_{\star}-v \\ \D_2
\end{array}\right] \nonumber \\
&& \hspace{2cm}+\frac{\tilde{\r}_{\star}}{2\g\tilde{u}_{\star}}(e^{\g h\tilde{u}_*t}-1)\left[\begin{array}{c}
0 \\ 0 \\ -(\tilde{u}_{\star}+v)\D_1-(\tilde{u}_{\star}-v)\D_2+2\tilde{u}_{\star}(\widetilde{H}_{\star}+v^2)-\frac{xq}{\tilde{\r}_{\star}}
\end{array}\right] \nonumber \\
\end{eqnarray}
where $\widetilde{H}_\star=H(\widetilde{U}_\star)$ and $t\in[0,\D t]$. Finally, replacing $\widetilde{U}_{\star}$ in \eqref{avgperturb} by $\widetilde{U}$,
the perturbation $\widebar{U}$ is obtained and is given in the form
\begin{equation}
\label{perturbsol}
\widebar{U}(x,t)=(S(x,t,\widetilde{U})-I_3)\widetilde{U},
\end{equation}
where
\begin{equation*}
S(x,t,\widetilde{U})=\left[\begin{array}{ccc}
e^{h\tilde{u}t}\cosh(hvt) & 0 & 0 \\
\vspace{-0.4cm} \\
ve^{h\tilde{u}t}\sinh(hvt) & e^{h\tilde{u}t}\cosh(hvt) & 0 \\
S_{31} & S_{32} & S_{33}
\end{array}\right],
\end{equation*}
with
\begin{eqnarray*}
&&S_{31}=\frac{-xq}{2\g\tilde{m}}(e^{\g h\tilde{u}t}-1)-\frac{v^3}{v^2-(\g-1)^2\tilde{u}^2}
(ve^{\g h\tilde{u}t}-ve^{h\tilde{u}t}\cosh(hvt)-(\g-1)\tilde{u}e^{h\tilde{u}t}\sinh(hvt)), \\
&&S_{32}=\frac{(\g-1)\tilde{u}(v^2+(\g-1)\tilde{u}^2)}{2(v^2-(\g-1)^2\tilde{u}^2)}
(e^{\g h\tilde{u}t}-e^{h\tilde{u}t}\cosh(hvt))
+\frac{v(2v^2-(3\g-2)(\g-1)\tilde{u}^2)}{2(v^2-(\g-1)^2\tilde{u}^2)}e^{h\tilde{u}t}\sinh(hvt), \\
&&S_{33}=-\frac{(\g-1)(v^2+(\g-1)\tilde{u}^2)}{v^2-(\g-1)^2\tilde{u}^2}e^{\g h\tilde{u}t}
+\frac{\g v}{v^2-(\g-1)^2\tilde{u}^2}(ve^{h\tilde{u}t}\cosh(hvt)+(\g-1)\tilde{u}e^{h\tilde{u}t}\sinh(hvt)).
\end{eqnarray*}

\section{Estimation of the residual \boldmath$\widehat{R}(U_{\D x},D_{k,n},\p)$}
\setcounter{equation}{0}

Let $(\eta,\om)$ be an entropy pair of \eqref{HEP1}. Here, we describe the estimation of the residual $\widehat{R}(U_{\D x},D_{k,n},\p)$
\begin{align*}
&\widehat{R}(U_{\Delta{x}},D_{k,n}, \phi)
:=\iint_{D_{k,n}}\left\{\eta({U}_{\D x})\phi_t+\om({U}_{\D x})\phi_x+d\eta({U}_{\D x})\cdot h(x)g(x,U_{\D x})\phi\right\}dxdt,
\end{align*}
where $U_{\D x},\ D_{k,n}$, and $\p$ are defined as in Theorem 3.5.

According to the Taylor expansion of $\eta({U}_{\Delta{x}})$ and \eqref{2.13.1}, for any positive test function $\phi\in C^1_0(\Pi)$,
\begin{align}
\label{b.3}
&\widehat{R}(U_{\Delta{x}}, D_{k,n}, \phi)\nonumber\\
&=\iint_{D_{k,n}}\left\{\eta(\widetilde{U}_{\D x})\phi_t+\om(\widetilde{U}_{\D x})\phi_x\right\}dxdt\nonumber\\
&\quad+\iint_{D_{k,n}}\big\{\big(\eta(U_{\D x})-\eta(\widetilde{U}_{\D x})\big)\phi_t+\big(\om(U_{\D x})-\om(\widetilde{U}_{\D x})\big)\phi_x
  +d\eta(U_{\D x})\cdot h(x)g(x,U_{\D x})\phi\big\} dxdt\nonumber\\
&=\iint_{D_{k,n}}\{\eta(\widetilde{U}_{\Delta{x}})\phi_t+\om(\widetilde{U}_{\Delta{x}})\phi_x\}dxdt\nonumber\\
&\quad+\iint_{D_{k,n}}\{\big(d\eta(\widetilde{U}_{\D x})\widebar{U}_{\D x}\big)\phi_t+\big(\om(U_{\D x})-\om(\widetilde{U}_{\D x})\big)\phi_x
  +d\eta(\widetilde{U}_{\D x})\cdot h(x)g(x,U_{\D x})\phi\} dxdt  \nonumber\\
  &\quad+O(1)(\Delta t)^3 \nonumber\\
&\equiv\widehat{Q}_{k,n}^1+\widehat{Q}_{k,n}^2+O(1)(\Delta t)^3.
\end{align}
In \cite{S2}, if $(\eta,\om)$ is an entropy pair, then
\begin{align}
\label{b.4}
\frac{\partial \eta(\widetilde{U}_{\D x}(x,t))}{\partial t}+\frac{\partial \om(\widetilde{U}_{\D x}(x,t))}{\partial x}\leq 0
\end{align} in the sense of distribution on the rectangle $D_{k,n}$.
Multiplying the left hand side of \eqref{b.4} by a positive test function $\phi$, and integrating by parts
over $D_{k,n}$, we obtain
\begin{align}
\label{b.5}
\widehat{Q}_{k,n}^1
&\geq\int^{x_{k+1}}_{x_{k-1}}\eta(\widetilde{U}_{\D x}(x, t))\p\Big|_{t=t_n^+}^{t=t_{n+1}^-}dx
 +\int^{t_{n+1}}_{t_n}\om(\widetilde{U}_{\D x}(x, t))\p\Big|_{x=x_{k-1}}^{x=x_{k+1}}dt.
\end{align}
The estimate of $\widehat{Q}_{k,n}^2$ is similar to that of $Q_2$ in Theorem 2.2. Let $(x_k,t_n)=(x_k,0)$, $D=D_{x_k,0}$, and let $\widetilde{U}_{\Delta x}$ in $D$ consist of the 1-shock with speed $s_1$, 2-contact discontinuity with speed $s_2$, and 3-rarefaction wave with lower speed $s_3^-$ and upper speed $s^+_3$. Then,
\begin{align}
\label{b.6}
\widehat{Q}_{k,n}^2
&=\left(\int^{\D t}_0\!\!\!\int^{x_k+s_1t}_{x_k-\D x}+\int_0^{\D t}\!\!\!\int_{x_k+s_1t}^{x_k+s_2t}+\int^{\D t}_0\!\!\!\int^{x_k+s^-_3t}_{x_k+s_2t}
 +\int^{\D t}_0\!\!\!\int_{x_k+s_3^-t}^{x_k+s_3^+t}+\int^{\D t}_0\!\!\!\int_{x_k+s_3^+t}^{x_k+\D x}\right)\nonumber\\
 &\qquad\Big\{\big(d\eta(\widetilde{U}_{\Delta{x}})\widebar{U}_{\Delta{x}}\big)\phi_t
 +\big(\om(U_{\Delta{x}})-\om(\widetilde{U}_{\Delta{x}})\big)\phi_x+d\eta(\widetilde{U}_{\D x})\cdot h(x)g(x,U_{\D x})\phi\Big\} dxdt\nonumber\\
 &\equiv \widehat{Q}^{21}+\widehat{Q}^{22}+\widehat{Q}^{23}+\widehat{Q}^{24}+\widehat{Q}^{25}.
\end{align} According to \eqref{diff1} and integration by parts, we have
\begin{align}
\label{b.7}
\widehat{Q}^{21}&=\int^{x_k+s_1\D t}_{x_k-\D x}\big(d\eta(\widetilde{U}_{\D x})\widebar{U}_{\D x}\p\big)(x,\Delta{t})dx-
\int^{x_k}_{x_k-\D x}\big(d\eta(\widetilde{U}_{\D x})\widebar{U}_{\D x}\p\big)(x,0)dx\nonumber\\
&\quad-\int^{\D t}_{0}s_1\big(d\eta(\widetilde{U}_{\D x})\widebar{U}_{\D x}\p\big)(s_1t-,t)dt
 +\int^{\D t}_{0}\big((\om(U_{\D x})-\om(\widetilde{U}_{\D x}))\p\big)(x,t)\Big|^{x=s_1t-}_{x=-\D x}dt\nonumber\\
&\quad+O(1)(\Delta t)^3(\Delta x)\|\phi\|_\infty, \\
\label{b.8}
\widehat{Q}^{22}&=\int^{x_k+s_2\Delta{t}}_{x_k+s_1\Delta{t}}\big(d\eta(\widetilde{U}_{\D x})\widebar{U}_{\D x}\p\big)(x,\Delta{t})dx
 +\int^{\D t}_{0}\big((\om(U_{\D x})-\om(\widetilde{U}_{\D x})-s_2d\eta(\widetilde{U}_{\D x})\widebar{U}_{\D x})\p\big)(s_2t-, t)dt\nonumber\\
&\quad-\int^{\D t}_{0}\big((\om(U_{\D x})-\om(\widetilde{U}_{\D x})-s_1d\eta(\widetilde{U}_{\D x})\widebar{U}_{\D x})\p\big)(s_1t+, t)dt
 +O(1)(\Delta t)^3(\Delta x)\|\phi\|_\infty, \\
\label{b.8.5}
\widehat{Q}^{23}&=\int^{x_k+s_3^-\D t}_{x_k+s_2\D t}\big(d\eta(\widetilde{U}_{\D x})\widebar{U}_{\D x}\p\big)(x,\Delta{t})dx
 +\int^{\D t}_{0}\big((\om(U_{\D x})-\om(\widetilde{U}_{\D x})-s_3^-d\eta(\widetilde{U}_{\D x})\widebar{U}_{\D x})\p\big)(s_3^-t, t)dt\nonumber\\
&\quad-\int^{\D t}_{0}\big((\om(U_{\D x})-\om(\widetilde{U}_{\D x})-s_2d\eta(\widetilde{U}_{\D x})\widebar{U}_{\D x})\p\big)(s_2t+, t)dt
 +O(1)(\Delta t)^3(\Delta x)\|\phi\|_\infty,
\end{align}
and
\begin{align}
\label{b.9}
\widehat{Q}^{25}&=\int^{x_k+\D x}_{x_k+s_3^+\D t}\big(d\eta(\widetilde{U}_{\D x})\widebar{U}_{\D x}\p\big)(x,\Delta{t})dx-
\int_{x_k}^{x_k+\D x}\big(d\eta(\widetilde{U}_{\D x})\widebar{U}_{\D x}\p\big)(x,0)dx\nonumber\\
&\quad+\int^{\D t}_{0}s_3^+\big(d\eta(\widetilde{U}_{\D x})\widebar{U}_{\D x}\p\big)(s_3^+t,t)dt
 +\int^{\D t}_{0}\big((\om(U_{\D x})-\om(\widetilde{U}_{\D x}))\p\big)(x, t)\Big|^{x=\D x}_{x=s_3^+t}dt\nonumber\\
&\quad+O(1)(\Delta t)^3(\Delta x)\|\phi\|_\infty.
\end{align}
Next, we estimate $\widehat{Q}^{24}$. According to the same argument as that for \eqref{diff2}-\eqref{Q242} and $d\om=d\eta df$, we obtain
\begin{align}
\label{b.10}
\widehat{Q}^{24}&=\int^{x_k+s_3^+\D t}_{x_k+s_3^-\D t}\big(d\eta(\widetilde{U}_{\D x})\widebar{U}_{\D x}\p\big)(x,\Delta{t})dx
 +\int^{\D t}_0\big((\om(U_{\D x})-\om(\widetilde{U}_{\D x})-s_3^+d\eta(\widetilde{U}_{\D x})\widebar{U}_{\D x})\p\big)(s_3^+t, t)dt\nonumber\\
&\quad-\int^{\D t}_0\big((\om(U_{\D x})-\om(\widetilde{U}_{\D x})-s_3^-d\eta(\widetilde{U}_{\D x})\widebar{U}_{\D x})\p\big)(s_3^-t, t)dt\nonumber\\
&\quad+O(1)\left((\D t)^3(\D x)+(\D t)^2(\underset{D_{k,n}}{\osc}\{\widetilde{U}\})\right)\|\phi\|_\infty.
\end{align}
By \eqref{b.6}-\eqref{b.10},
\begin{align}
\label{b.11}
\widehat{Q}^2_{k,n}&=\sum_{i=1}^5\widehat{Q}^{2i}
  =\int^{\D x}_{-\D x}\big(d\eta(\widetilde{U}_{\D x})\widebar{U}_{\D x}\p\big)(x,t)\Big|_{t=0}^{t=\D t}dx
  +\int^{\D t}_0\big((\om(U_{\D x})-\om(\widetilde{U}_{\D x}))\p\big)(x, t)\Big|_{x=-\D x}^{x=\D x}dt\nonumber\\
&\quad+\int^{\D t}_0\big((\om(U_{\D x})-\om(\widetilde{U}_{\D x})
  -s_1d\eta(\widetilde{U}_{\D x})\widebar{U}_{\D x})\p\big)(x, t)\Big|^{x=s_1t-}_{x=s_1t+}dt\nonumber\\
&\quad+\int^{\D t}_0\big((\om(U_{\D x})-\om(\widetilde{U}_{\D x})
  -s_2d\eta(\widetilde{U}_{\D x})\widebar{U}_{\D x})\p\big)(x, t)\Big|^{x=s_2t-}_{x=s_2t+}dt
\nonumber\\
&\quad+O(1)\left((\Delta t)^3(\Delta x)+(\Delta{t})^3+(\Delta{t})^2(\underset{D}{\osc}\{\widetilde{U}\})\right)\|\phi\|_\infty.
\end{align}

We estimate the second and third terms on the right-hand side of \eqref{b.11}. Suppose that the state $\widetilde{U}_1$ is connected to the state $\widetilde{U}_L=U_L=(\r_L,m_L,E_L)$ by 1-shock on the right, and to the state $\widetilde{U}_2$ by 2-contact discontinuity on the left. By the result of \cite{S2}, \eqref{RHcond} is replaced by
\begin{equation}
\label{b.12}
\om(\widetilde{U}_1)-\om(\widetilde{U}_L)\ge s_1(\widetilde{U}_1-\widetilde{U}_L),\quad
\om(\widetilde{U}_1)-\om(\widetilde{U}_2)\ge s_2(\widetilde{U}_1-\widetilde{U}_2).
\end{equation}
Therefore, according to \eqref{b.12} and similar argument as \eqref{solverestmate}, we obtain
\begin{align}
\label{b.13}
&\int_0^{\D t}\big((\om(U_{\D x})-\om(\widetilde{U}_{\D x})-s_1d\eta(\widetilde{U}_{\D x})\widebar{U}_{\D x})\p\big)(x,t)\Big|_{x=s_1t+}^{x=s_1t-}dt
 \nonumber\\
&\hspace{3cm}\ge\Big(O(1)(\D t)^2(\underset{D}{\osc}\{\widetilde{U}\})+O(1)\D t(\underset{D}{\osc}\{\widetilde{U}\})^2\Big)\|\p\|_{\infty}, \\
\label{b.14}
&\int_0^{\D t}\big((\om(U_{\D x})-\om(\widetilde{U}_{\D x})-s_2d\eta(\widetilde{U}_{\D x})\widebar{U}_{\D x})\p\big)(x,t)\Big|_{x=s_2t+}^{x=s_2t-}dt
 \nonumber\\
&\hspace{3cm}\ge\Big(O(1)(\D t)^2(\underset{D}{\osc}\{\widetilde{U}\})+O(1)(\D t)(\underset{D}{\osc}\{\widetilde{U}\})^2\Big)\|\p\|_{\infty}.
\end{align}
According to \eqref{b.3},\ \eqref{b.5},\ \eqref{b.11},\ \eqref{b.13}, and \eqref{b.14}, we obtain the result of \eqref{RPres1}:
\begin{align*}
&\widehat{R}(U_{\D x},D_{k,n},\phi)\ge\int^{x_{k+1}}_{x_{k-1}}\big(d\eta(\widetilde{U}_{\D x})\widebar{U}_{\D x}\p\big)(x,t^-_{n+1})dx
-\int^{x_{k+1}}_{x_{k-1}}\big(\eta(\widetilde{U}_{\D x})\p\big)(x,t)\Big|^{t=t^-_{n+1}}_{t=t^+_n}dx\nonumber\\
&\quad+\int^{t_{n+1}}_{t_n}\big(\om(U_{\D x})\p\big)(x,t)\Big|^{x=x_{k+1}}_{x=x_{k-1}}dt
+O(1)\left((\Delta t)^2(\Delta x)+(\Delta{t})^3+(\Delta{t})^2\underset{D_{k,n}}{\osc}\{\widetilde{U}\}\right)\|\phi\|_\infty.
\end{align*}

\vskip 1cm





\begin{thebibliography}{10}




\bibitem{BB}
{\sc S. Bianchini, A. Bressan},
{\em Vanishing viscosity solutions of nonlinear hyperbolic systems},
Ann. of Math., 161 (2005), pp. 223--342.

\bibitem{Chamberlain1987}
{\sc J.W. Chamberlain and D.M. Hunten},
{\em Theory of Planetary Atmospheres},
Academic, Orlando, Fla (1987).

\bibitem{CSW}
{\sc G.-Q. Chen, M. Slemrod, D. Wang},
{\em Vanishing viscosity method for transonic flow},
Arch. Rational Mech. Anal., 189 (2008), pp. 159--188.

\bibitem{CHS}
{\sc S.-W. Chou, J. M. Hong, Y.-C. Su},
{\em An extension of Glimm's method to the gas dynamical model of transonic flows},
Nonlinearity, 26 (2013), pp. 1581--1597.

\bibitem{CHS1}
{\sc S.-W. Chou, J. M. Hong, Y.-C. Su},
{\em Global entropy solutions of the general nonlinear hyperbolic balance laws with time-evolution flux and source},
Mathods Appl. Anal., 19 (2012), pp. 43--76.

\bibitem{CHS2}
{\sc S.-W. Chou, J. M. Hong, Y.-C. Su},
{\em The initial-boundary value problem of hyperbolic integro-differential systems of nonlinear balance laws},
Nonlinear Analysis: Theory, Methods and Applications, 75 (2012), pp. 5933--5960.

\bibitem{DH}
{\sc C. M. Dafermos, L. Hsiao},
{\em Hyperbolic systems of balance laws with inhomogeneity and dissipation},
Indiana Univ. Math. J., 31 (1982), pp. 471--491.

\bibitem{DLM}
{\sc G. Dal Maso, P. LeFloch, F. Murat},
{\em Definition and weak stability of nonconservative products},
J. Math. Pure Appl., 74 (1995), pp. 483--548.

\bibitem{Erwin2013}
{\sc J. Erwin, O.J. Tucker and R.E. Johnson},
{\em Hybrid fluid/kinetic modeling of Pluto's escaping atmosphere},
Icarus, 226 (2013), pp. 375--384.

\bibitem{G}
{\sc J. Glimm},
{\em Solutions in the large for nonlinear hyperbolic systems of equations},
Commun. Pure Appl. Math., 18 (1965), pp. 697--715.

\bibitem{G1}
{\sc J. B. Goodman},
{\em Initial boundary value problems for hyperbolic systems of conservation laws},
Thesis (Ph. D.)--Stanford University., (1983).

\bibitem{GL}
{\sc P. Goatin, P.G. LeFloch},
{\em The Riemann problem for a class of resonant nonlinear systems of balance laws},
Ann. Inst. H. Poincare-Analyse Non-lineaire, 21 (2004), pp. 881--902.

\bibitem{GST}
{\sc J. Groah, J. Smoller, B. Temple},
{\em Shock Wave Interactions in General Relativity},
Monographs in Mathematics, Springer, Berlin, New York, 2007.

\bibitem{Guo2005}
{\sc Y. Guo, R.W. Farquhar},
{\em New horizons Pluto-Kuiper belt mission: Design and simulation of the Pluto-Charon encounter},
Acta Astro., 56 (2005), pp. 421-429.

\bibitem{RH2004}
{\sc James R. Holton},
{\em An introduction to dynamical meteorology},
Elsevier Academic Press 200, 4th ed., Wheeler Road, Burlington, MA 01803, USA, 2004.

\bibitem{H}
{\sc J. M. Hong},
{\em An extension of Glimm's method to inhomogeneous strictly hyperbolic systems of conservation laws by ``weaker than weak'' solutions of the Riemann problem},
J. Diff. Equ., 222 (2006), pp. 515--549.

\bibitem{HL}
{\sc J. M. Hong, P.G. LeFloch},
{\em A version of Glimm method based on generalized Riemann problems},
J. Portugal Math., 64, (2007) pp. 199--236.

\bibitem{HT1}
{\sc J. M. Hong, B. Temple},
{\em The generic solution of the Riemann problem in a neighborhood of a point of resonance for systems of nonlinear balance laws},
Methods Appl. Anal., 10 (2003), pp. 279--294.

\bibitem{HT2}
{\sc J. M. Hong, B. Temple},
{\em A bound on the total variation of the conserved quantities for solutions of a general resonant nonlinear balance law},
SIAM J. Appl. Math., 64 (2004), pp. 819--857.

\bibitem{HYH}
{\sc J. M. Hong, C.-C. Yen, B.-C. Huang},
{\em Characterization of the transonic stationary solutions of the hydrodynamic escape problem},
SIAM J. Appl. Math., 74 (2014), pp. 1709--1741.

\bibitem{Hunten1982}
{\sc D.M. Hunten},
{\it Planet},
Space Sci., 30 (1982), pp. 773-783.

\bibitem{Hunten1987}
{\sc D.M. Hunten, R.O. Pepin, J.C.G. Walker},
{\em Mass fractionation in hydrodynamic escape},
Icarus, 69 (1987), pp. 532-549.

\bibitem{IT1}
{\sc E. Isaacson, B. Temple},
{\em Nonlinear resonance in systems of conservation laws},
SIAM J. Appl. Anal., 52 (1992), pp. 1260--1278.

\bibitem{IT2}
{\sc E. Isaacson, B. Temple},
{\em Convergence of the $2\times2$ Godunov method for a general resonant nonlinear balance law},
SIAM J. Appl. Math., 55, (1995), pp. 625--640.

\bibitem{LA}
{\sc P. D. Lax},
{\em Hyperbolic system of conservation laws II},
Commun. Pure Appl. Math., 10 (1957), pp. 537--566.

\bibitem{LF1}
{\sc P. G. LeFloch},
{\em Entropy weak solutions to nonlinear hyperbolic systems under nonconservative form},
Commun. Part. Diff. Equ., 13 (1988) pp. 669--727.

\bibitem{LF2}
{\sc P. G. LeFloch},
{\em Shock waves for nonlinear hyperbolic systems in nonconservative form},
Institute for Math. and its Appl., Minneapolis, Preprint 593, 1989.

\bibitem{LL}
{\sc P. G. LeFloch, T.-P. Liu},
{\em Existence theory for nonlinear hyperbolic systems in nonconservative form},
Forum Math., 5 (1993), pp. 261--280.

\bibitem{LR}
{\sc P. G. LeFloch, P. A. Raviart},
{\em Asymptotic expansion for the solution of the generalized Riemann problem, Part 1},
Ann. Inst. H. Poincare, Nonlinear Analysis, 5 (1988) pp. 179--209.

\bibitem{LE}
{\sc Randall J. LeVeque},
{\em Finite Volume Methods for Hyperbolic Problem},
Cambridge Texts in Applied Mathematics (2002), Cambridge.

\bibitem{LIANG}
{\sc Mao-Chang Liang, Alan N. Heays, Brenton R. Lewis, Stephen T. Gibson, and Yuk L. Yung},
{\em Source of Nitrogen Isotope Anmaly in HCN in the Atmosphere ot Titan},
The Astrophysical Journal, 664 (2007), pp. 115-118.

\bibitem{Lin2012}
{\sc R.P. Lin, B. Jakosky},
{\em The 2013 Mars Atmosphere and Volatile Evolution (MAVEN) Mission to Mars. In: 39th COSPAR Scientific Assembly},
COSPAR Meeting, 39 (2012), pp. 1089.

\bibitem{TP1}
{\sc T.-P. Liu},
{\em Quasilinear hyperbolic systems},
Commun. Math. Phys., 68 (1979), pp. 141--172.

\bibitem{TP2}
{\sc T.-P. Liu},
{\em Nonlinear stability and instability of transonic flows through a nozzle},
Commun. Math. Phys., 83 (1982), pp. 243--260.

\bibitem{TP3}
{\sc T.-P. Liu},
{\em Nonlinear resonance for quasilinear hyperbolic equation},
J. Math. Phys., 28 (1987), pp. 2593--2602.

\bibitem{LT}
{\sc M. Luskin, B. Temple},
{\em The existence of global weak solution to the nonlinear waterhammer problem},
Commun. Pure Appl. Math., 35 (1982), pp. 697--735.

\bibitem{M}
{\sc C. S. Morawetz},
{\em On a weak solution for a transonic flow problem},
Commun. Pure Appl. Math., 38 (1985), pp. 797--817.

\bibitem{MurrayClay2009}
{\sc R.A. Murray-Clay, E.I. Chiang and N. Murray},
{\em Atmospheric Escape From Hot Jupiters},
The Astrophysical Journal, 693 (2009), pp. 23-43.

\bibitem{Parker1964a}
{\sc E.N. Parker},
{\em Dynamical properties of stellar coronas and stellar winds I. Integration of the momentum equation},
Astrophys. J., 139 (1964), pp. 72-92.


\bibitem{Parker1964b}
{\sc E.N. Parker},
{\em Dynamical properties of stellar coronas and stellar winds II. Integration of the heat-flow equation},
Astrophys. J., 139 (1964), pp. 93-122.

\bibitem{S1}
{\sc J. Smoller},
{\em On the solution of the Riemann problem with general stepdata for an extended class of hyperbolic system},
Mich. Math. J., 16, pp. 201--210.

\bibitem{S2}
{\sc J. Smoller},
{\em Shock Waves and Reaction-Diffusion Equations},
2nd ed., Springer-Verlag, Berlin, New York, 1994.

\bibitem{ST2008A}
{\sc Darrell F. Strobel},
{\em Titan's hydrodynamically escaping atmosphere},
Icarus, 193 (2008), pp. 588--594.

\bibitem{ST2008B}
{\sc Darrell F. Strobel},
{\em $N_2$ escape rates from Pluto's atmosphere},
Icarus, 193 (2008), pp. 612--619.

\bibitem{TE1}
{\sc B. Temple},
{\em Global solution of the Cauchy problem for a class of 2$\times$2 nonstrictly hyperbolic conservation laws},
Adv. Appl. Math., 3 (1982), pp. 335--375.

\bibitem{Tian2008}
{\sc F. Tian, J.F. Kasting, H.L. Liu, R.G. Roble},
{\em Hydrodynamic planetary thermosphere model 1: Response of the earth's thermosphere to extreme solar euv conditions and the significance of adiabatic cooling},
J. Geophys. Res.: Planets, 113 (2008), E05008.

\bibitem{TI2005B}
{\sc Feng Tian, Owen B. Toon, Alexander A. Pavlov, and H. De Sterck},
{\em Transonic hydrodynamic escape of hydrogen from extrasolar planetary atmospheres},
The Astrophys. J., 621 (2005), pp. 1049--1060.

\bibitem{TI2005A}
{\sc Feng Tian, Owen B. Toon, Alexander A. Pavlov, and H. De Sterck},
{\em A hydrogen-rich early Earth atmosphere},
Scicence, 308 (2005), pp. 1014--1017.

\bibitem{TEDVJ2012}
{\sc O. J. Tucker, J. T. Erwin, J. I. Deghan, A. N. Volkov, and R. E. Johnson},
{\em Thermally driven escape from Pluto's atmosphere: A combined fluid/kinetic model},
Icarus, 217 (2012), pp. 408--415.


\bibitem{V2011}
{\sc A. N. Volkov, R. E. Johnson, O. J. Tucker, and J. T. Erwin},
{\em Thermally driven atmospheric escape: transition from hydrodynamic to Jeans escape},
The Atrophys. J. Lett., 729:L24 (5pp) (2011).

\bibitem{Y2004}
{\sc R.V. Yelle},
{\em Aeronomy of extra-solar giant planets at small orbital distances},
Icarus, 170 (2004), pp. 167--179.

\bibitem{ZSE}
{\sc Xum Zhu, Darrell F. Strobel, and Justin T. Erwin},
{\em The density and thermal structure of Pluto's atmosphere and associated escape processes and rates},
Icarus, 228 (2014), pp. 301--314.


 \end{thebibliography}
\end{document}